%% file: MainFile.tex
\documentclass[10pt]{amsart}

\usepackage{anysize}
\usepackage[utf8x]{inputenc}
\usepackage[english]{babel}
\usepackage{amssymb, amsmath, amsthm}
\usepackage[shortlabels]{enumitem}
\usepackage{mathtools}
\usepackage{float}
\usepackage{tikz-cd}
\usetikzlibrary{arrows}
\usepackage{url}
\usepackage[all,cmtip]{xy}
\usepackage[hidelinks,bookmarksnumbered]{hyperref}

\author{Ruben Henrard}
\address{Ruben Henrard \\ Universiteit Hasselt \\ Campus Diepenbeek \\ Departement WNI \\ 3590 Diepenbeek \\ Belgium}
\email{ruben.henrard@uhasselt.be}
\author{Adam-Christiaan van Roosmalen}
\address{Adam-Christiaan van Roosmalen \\ Universiteit Hasselt \\ Campus Diepenbeek \\ Departement WNI \\ 3590 Diepenbeek \\ Belgium}
\email{adamchristiaan.vanroosmalen@uhasselt.be}
\title{A categorical framework for glider representations}

\newtheorem{theorem}{Theorem}[section]
\newtheorem{proposition}[theorem]{Proposition}
\newtheorem{lemma}[theorem]{Lemma}
\newtheorem{corollary}[theorem]{Corollary}

\theoremstyle{definition}
\newtheorem{definition}[theorem]{Definition}
\newtheorem{remark}[theorem]{Remark}
\newtheorem{example}[theorem]{Example}

\newtheorem{notation}[theorem]{Notation}

\newtheorem{descript}[theorem]{Description}

\let\frak\mathfrak
\def\aa{{\frak a}}
\def\bb{{\frak b}}

\let\cal\mathcal
\def\AA{{\cal A}}
\def\BB{{\cal B}}
\def\CC{{\cal C}}
\def\DD{{\cal D}}
\def\EE{{\cal E}}
\def\FF{{\cal F}}

\def\MM{{\cal M}}
\def\NN{{\cal N}}

\def\TT{{\cal T}}

\def\VV{{\cal V}}

\def\XX{{\cal X}}

\def\bR{\mathbb{R}}
\def\bC{\mathbb{C}}
\def\bZ{\mathbb{Z}}

\DeclareMathOperator{\im}{im}
\DeclareMathOperator{\coim}{coim}
\DeclareMathOperator{\coker}{coker}

\DeclareMathOperator{\Hom}{Hom}
\DeclareMathOperator{\Fun}{Fun}
\DeclareMathOperator{\Mor}{Mor}
\DeclareMathOperator{\End}{End}

\DeclareMathOperator{\Mod}{Mod}

\DeclareMathOperator{\Proj}{Proj}
\DeclareMathOperator{\Inj}{Inj}

\DeclareMathOperator{\Ab}{\mathsf{Ab}}
\DeclareMathOperator{\Ob}{Ob}
\DeclareMathOperator{\C}{\mathbf{C}}
\DeclareMathOperator{\Cb}{\mathbf{C}^b}
\DeclareMathOperator{\K}{\mathbf{K}}
\DeclareMathOperator{\Kb}{\mathbf{K}^b}

\DeclareMathOperator{\Db}{\mathbf{D}^b}
\DeclareMathOperator{\DAb}{\mathbf{D}_{\AA}^b}
\DeclareMathOperator{\DModRb}{\mathbf{D}_{\Mod(R)}^b}
\DeclareMathOperator{\Ac}{\mathbf{Ac}}
\DeclareMathOperator{\Acb}{\mathbf{Ac}^b}

\DeclareMathOperator{\colim}{colim}

\newcommand{\oFF}{\overline{\FF}}
\renewcommand{\mod}{\operatorname{mod}}
\DeclareMathOperator{\Preglid}{Preglid_\Lambda}
\DeclareMathOperator{\Prefrag}{Prefrag_\Lambda}
\DeclareMathOperator{\NPreglid}{NPreglid_\Lambda}
\DeclareMathOperator{\NGlid}{NGlid_\Lambda}
\DeclareMathOperator{\preglid}{preglid_\Lambda}
\DeclareMathOperator{\prefrag}{prefrag_\Lambda}
\DeclareMathOperator{\Glid}{Glid_\Lambda}
\DeclareMathOperator{\glid}{glid_\Lambda}
\DeclareMathOperator{\Frag}{Frag_\Lambda}

\DeclareMathOperator{\Vect}{vec}

\newcommand{\inflation}{\rightarrowtail}
\DeclareMathOperator{\deflation}{\twoheadrightarrow}

\begin{document}


\begin{abstract}
Glider representations are a generalization of filtered modules over filtered rings.  Given a $\Gamma$-filtered ring $FR$ and a subset $\Lambda \subseteq \Gamma$, we provide a category $\Glid FR$ of glider representations, and show that it is a complete and cocomplete deflation quasi-abelian category.  We discuss its derived category, and the subcategories of natural gliders and Noetherian gliders.

If $R$ is a bialgebra over a field $k$ and $FR$ is a filtration by bialgebras, we show that $\Glid FR$ is a monoidal category which is derived equivalent to the category of representations of a semi-Hopf category (in the sense of Batista, Caenepeel, and Vercruysse).  We show that the monoidal category of glider representations associated to the one-step filtration $k \cdot 1 \subseteq R$ of a bialgebra $R$ is sufficient to recover the bialgebra $R$ by recovering the usual fiber functor from $\Glid FR.$  When applied to group algebras, this shows that the monoidal category $\Glid F(kG)$ alone is sufficient to distinguish even isocategorical groups.
\end{abstract}

\maketitle

\tableofcontents

\input{Introduction}
\input{Preliminaries}
\input{CategoryGlider}
\input{StructureResults}
\input{MainProperties}
\input{Naturalgliders}
\input{DerivedRecollement}
\input{BialgebraFiltration}
\appendix
\input{Appendix}

\bibliographystyle{amsplain}
\providecommand{\bysame}{\leavevmode\hbox to3em{\hrulefill}\thinspace}
\providecommand{\MR}{\relax\ifhmode\unskip\space\fi MR }
\providecommand{\MRhref}[2]{%
  \href{http://www.ams.org/mathscinet-getitem?mr=#1}{#2}
}
\providecommand{\href}[2]{#2}

\end{document}

%% file: Introduction.tex
\section{Introduction}

Filtered rings and their representations occur naturally in many areas of mathematics and physics, for example in the study of $D$-modules and quantum groups (see, for example, \cite{AsensioVandenBerghVanOstaeyen89, Kashiwara03, McConnell90, NastasescuVanOystaeyen79}). The following example illustrates that sometimes one is interested in truncated filtered representations.  Consider the $\mathbb{Z}$-filtered ring $\mathbb{R}[t]$ with $\deg(t)=1$.  For each $i\geq 0$, write $C^{i}(\mathbb{R})$ for the $i$-times continuously differentiable functions on $\mathbb{R}$, and consider the chain 
\[\dots\subseteq C^{i+1}(\mathbb{R}) \subseteq C^{i}(\mathbb{R})\subseteq \dots \subseteq C^{1}(\mathbb{R})\subseteq C^0(\mathbb{R}).\]
Whenever $i\geq n$, $t^n$ induces an action $t^n\colon C^{i}(\mathbb{R})\to C^{i-n}(\mathbb{R})\colon f\mapsto \frac{\mathrm{d}^nf}{\mathrm{d}x^n}$. In this way, the differential operator $\frac{\mathrm{d}}{\mathrm{d}x}$ (which is only densely defined on $C^0(\mathbb{R})$) is encoded algebraically into this filtration.

This type of truncated filtered representations were studied by Nawal and Van Oystaeyen in 1995 (see \cite{NawalVanOystaeyen95}) under the name of \emph{fragments} (as the definition of fragment has changed over the years, we will refer to this concept as a \emph{prefragment}).  Given a ring $R$ and a positive ring filtration $S=F_0R\subseteq F_1R\subseteq F_2R\subseteq \dots \subseteq F_nR \subseteq \dots $ one can consider an $FR$-fragment $M$, i.e. an $\bZ^{\leq 0}$-filtered $S$-module $M$ together with $S$-submodules
\[\ldots \subseteq M_{-i} \subseteq M_{-i+1} \subseteq \ldots \subseteq M_0 = M\]
and ``fragmented actions'' $\phi\colon F_n R \times M_{-n} \to M$.  Going deeper into the chain, one gradually sees more of the $R$-action.  The precise definition is recalled in definition \ref{definition:Fragment}.

Looking at the definition of an $FR$-prefragment, it is not clear to what extent one should require the fragmented actions to be compatible with each other.  Over the course of several papers, the original definition of an $FR$-prefragment has been amended, requiring more associativity conditions on the partial actions (compare, for example, the definition in \cite{NawalVanOystaeyen95} to the one in \cite{CaenepeelVanOystaeyenBook19}).  One way to circumvent the associativity issue is to require that all partial actions of an $FR$-prefragment $M$ are induced by some enveloping $R$-module $\Omega_M\supseteq M$. Prefragments satisfying this additional condition are called \emph{glider representations}. Example \ref{Example:IntroExample} below shows that $\{C^i(\mathbb{R})\}_{i \leq 0}$ is, in fact, a glider representation.

Recently, the theory of glider representations has regained some attention (see for example \cite{Caenepeel18, CaenepeelVanOystaeyen16, CaenepeelVanOystaeyen18} and the book \cite{CaenepeelVanOystaeyenBook19}).  Despite these new developments, a categorical framework for glider representations is missing (see \cite{CaenepeelVanOystaeyenBook19} or remark \ref{remark:Obstruction}).

We construct a category of glider representations over a filtered ring $FR$ as a localization of the category of pregliders.  For the purpose of this introduction, we sketch the construction of the categories of glider and preglider representations.  Let $\Gamma$ be an ordered group and fix any subset $\Lambda \subseteq \Gamma.$  Let $FR$ be a $\Gamma$-filtration of a ring $R$ and let $\Omega$ be any $R$-module.  We choose subgroups $\{M_\lambda\}_{\lambda \in \Lambda}$ of $\Omega$ such that the module action $R \times \Omega \to \Omega$ restricts to fragmented actions $F_{\lambda\mu^{-1}}R \times M_{\mu} \to M_{\lambda}$ (for all $\lambda, \mu \in \Lambda$).  We refer to $\Omega$ together with the subgroups $\{M_\lambda\}_{\lambda \in \Lambda}$ as a preglider.  A morphism $(\{M_\lambda\}_{\lambda \in \Lambda}, \Omega_M) \to (\{N_\lambda\}_{\lambda \in \Lambda}, \Omega_N)$ between pregliders is an $R$-module morphism $f\colon \Omega_M \to \Omega_N$ such that $f(M_\lambda) \subseteq N_\lambda$, for each $\lambda \in \Lambda.$

We then define the category of glider representations as the localization $(\Preglid FR)[\Sigma^{-1}]$ where $\Sigma$ consists of those morphisms such that the all induced maps $f_\lambda\colon M_\lambda \to N_\lambda$ are isomorphisms (but $f\colon M \to N$ itself need not be an isomorphism).  The set $\Sigma$ is a right multiplicative system in $\Preglid FR,$ so that the localization can be described using roofs (see proposition \ref{proposition:SigmaRMS}).

The category $\Prefrag FR$ of prefragments does, in general, not recover sufficient structure of the filtered ring: this is illustrated in example \ref{example:ShortPrefragIsBlind}.  In contrast, the category $\Preglid FR$ contains too much information: we see from example \ref{example:UsualFiltrations} that the category of pregliders does not reduce to the usual category of filtered modules.  The category $\Glid FR$ of glider representations lies in between the pregliders and the prefragments in the following way.  We can consider the restriction map $j^*\colon \Preglid FR \to \Prefrag FR$ which forgets the ambient $R$-module $\Omega.$  The category $\Glid FR$ is obtained from the usual factorization of $j^*\colon \Preglid FR \to \Glid FR \to \Prefrag FR$ where $\Glid FR = \Sigma^{-1}\Preglid FR$ and $\Sigma$ consists of all morphisms $\sigma$ between pregliders such that the restriction $j^*(\sigma)$ is invertible.

The following theorem provides different interpretations of the category of glider representations (see proposition \ref{proposition:PhiHasLeftAdjoint} and corollary \ref{corollary:GlidIsCoreflectiveInPreGlid}).

\begin{theorem}
Let $\Gamma$ be an ordered group and let $\Lambda \subseteq \Gamma.$
\begin{enumerate}
	\item The category $\Glid FR$ is a reflective subcategory of $\Prefrag FR,$ and
	\item the category $\Glid FR$ is a coreflective subcategory of $\Preglid FR.$
\end{enumerate}
\end{theorem}

The second statement implies that a glider representation admits a canonical ambient $R$-module.  The first statement indicates that a glider representation is a prefragment satisfying additional properties, rather than possessing additional structure.  In proposition \ref{proposition:WhenGlider}, we provide some criterion to decide whether a preframent is a glider representation.

In order to describe the rich structure of the category of glider representations, we will describe the relations with several other categories.  Figure \ref{figure:TheDiagram} below showcases the categories involved as well as some adjunctions.  Here, the preadditive category $\FF_\Lambda R$ is given by taking as objects the set $\Lambda$; the Hom-spaces are given by the the filtered ring $FR$.  Likewise, $\oFF_\Lambda R$ is given by adjoining an extra object with $R$ as endomorphism ring (see definition \ref{definition:CompanionCategory} for details).  We refer to $\FF_\Lambda R$ and $\oFF_\Lambda R$ as the \emph{filtered companion category} and the \emph{extended filtered companion category}.

\begin{figure}[htbp]
\[\xymatrix{
{\Mod R}\ar@{=}[d] \ar[rr]|{i_*} && {\Mod \oFF_\Lambda R} \ar[rr]|{j^*} \ar@/^/[ll]^{i^!} \ar@/_/[ll]_{i^*} \ar@/_/[d]_{\kappa} && {\Mod \FF_\Lambda R} \ar@/^/[ll]^{j_*} \ar@/_/[ll]_{j_!} \ar@/_/[d]_{\theta}\\
{\Mod R} \ar[rr]|{i_*} && {\Preglid FR} \ar[rr]|{j^*} \ar@/^/[ll]^{i^!} \ar@/_/[ll]_{i^*} \ar@/_/@{^{(}->}[u]_{\iota} \ar@/_/[drr]_{Q} && {\Prefrag FR} \ar@/_/@{^{(}->}[u]_{\eta} \ar@/_/[d]_{\psi} \ar@/_/[ll]_{j_L}\\
&&&&{\Glid FR} \ar@/_/@{^{(}->}[u]_{\phi} \ar@/_/[llu]_L 
}\]
	\caption{Overview diagram}
	\label{figure:TheDiagram}
\end{figure}
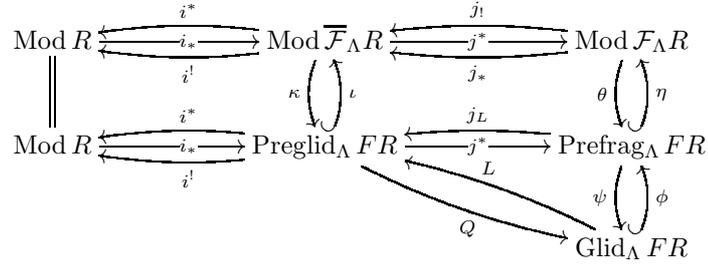

The categories $\Preglid FR$ and $\Prefrag FR$ are torsion-free classes in $\Mod \oFF_\Lambda R$ and $\Mod \FF_\Lambda R,$ respectively (proposition \ref{proposition:PreglidIsTorsionfree}).  As such, they are quasi-abelian categories (that is, they admit kernels and cokernels, and the class of all kernel-cokernel pairs endows them with the structure of a Quillen exact category).  In contrast, the category $\Glid FR$ has kernels and cokernels, but the class of kernel-cokernel pairs of $\Glid FR$ is not closed under pushouts (but it is closed under pullbacks).  Such a structure has been called a \emph{left almost abelian category} in \cite{Rump01} (we will refer to it as a \emph{deflation quasi-abelian category}).

The following theorem is shown in \S\ref{section:PropertiesOfTheCategoryOfGliderRepresentations}. 

\begin{theorem}\label{theorem:GlidersAreGrothendieckDeflationQuasiAbelian}
\begin{enumerate}
	\item The categories $\Preglid FR$ and $\Prefrag FR$ are complete and cocomplete quasi-abelian categories.
	\item The category $\Glid FR$ is a complete and cocomplete deflation quasi-abelian category.
\end{enumerate}
\end{theorem}

In fact, one shows that these categories are Grothendieck (deflation) quasi-abelian (see \S\ref{subsection:GrothendieckQuasiAbelian}). 

The conflation structure on $\Glid FR$ is compatible with the localization $Q\colon \Preglid FR \to \Glid FR$ and the embedding $\phi\colon \Glid FR \to \Prefrag FR$ in the following sense: it is the minimal conflation structure such that $Q$ preserves conflations, and the maximal conflation structure such that $\phi$ preserves conflations (see corollary \ref{corollary:GlidersAreDeflationExact} and proposition \ref{proposition:GlidAsSubcatOfPrefrag}).

In \S\ref{section:PropertiesOfTheCategoryOfGliderRepresentations}, we will also use the following interpretation of the category $\Glid FR$: it is given as the quotient $\Preglid FR / i_*(\Mod R)$ in the sense of \cite{HenrardVanRoosmalen19a}.  This is an example of a one-sided exact quotient of a two-sided exact category.

As $\Glid FR$ is a deflation quasi-abelian category, it admits a meaningful (bounded) derived category $\Db(\EE)$ (see \cite{BazzoniCrivei13, HenrardVanRoosmalen19b}).  Using the main result of \cite{HenrardVanRoosmalen19b}, we obtain a Verdier localization sequence
\[\Db(\Mod(R))\to \Db(\Preglid(FR))\to \Db(\Glid(FR)).\]
Moreover, in section \ref{TheDerivedCategoryOfGliderRepresentationsAsAVerdierLocalization}, we show that the categories in each column of figure \ref{figure:TheDiagram} are derived equivalent. Hence, the above Verdier localization sequence is equivalent to the sequence \[\Db(\Mod(R))\to \Db(\Mod(\oFF_{\Lambda} R))\to \Db(\Mod(\FF_{\Lambda}R)).\]
In particular, as in \cite{SchapiraSchneiders16} we find that $\Db(\Mod(\oFF_\Lambda R))\simeq \Db(\Glid(FR))$ are triangle equivalent.

In section \S\ref{section:HopfCategories} we consider a filtration $FB$ of a $k$-bialgebra $B$ such that each $F_nB$ is itself a $k$-coalgebra. For these type of filtrations the companion categories, $\FF_\Lambda R$ and $\oFF_\Lambda R$, are \emph{semi-Hopf categories} in the sense of \cite{BatistaCaenepeelVercruysse16,CaenepeelFieremans18}. It follows that the category $\Glid FB$ is a monoidal category. It is well-known that $\Mod B$ alone is not sufficient to reconstruct the bialgebra $B$.  On the other hand, the following theorem indicates that the monoidal category $\Glid FB$ (or, as we will assume that $B$ is finite-dimensional, the full subcategory $\glid FB$ consisting of the noetherian objects) contains enough information to reconstruct the bialgebra $B$.

\begin{theorem}\label{theorem:ReconstructionResultIntro}
	Let $B$ and $B'$ be bialgebras and consider the trivial filtrations $k\cdot 1_B\subseteq B$ and $k\cdot 1_{B'}\subseteq B'$. The categories $\glid(FB)$ and $\glid(FB')$ are monoidally equivalent if and only if $B$ and $B'$ are isomorphic as bialgebras.
\end{theorem}

In particular, the above theorem allows to distinguish even isocategorical groups (see \cite{EtingofGelaki01}) from the monoidal structure of the category $\glid F(kG)$ (without referring to the symmetric structure). This provides a conceptual explanation as to why the generalized character ring (which is related to the representation ring of the category $\glid F(kG)$) discussed in \cite{CaenepeelJanssens19} is capable of recovering more of the group structure than the ordinary character ring.

\subsection*{Structure of the paper}

We now turn to an overview of the paper.  Section \S\ref{section:Preliminaries} is preliminary in nature.  We recall some definitions and results that will be used throughout the paper.

In section \S\ref{section:GliderCategory} we construct the category $\Glid FR$ of glider representations.  We consider general $\Gamma$-filtered rings (where $\Gamma$ is a filtered group) and consider $\Lambda\subseteq \Gamma$.  We encode the information of $\Lambda$ and of the $\Gamma$-filtered ring $FR$ in the componanion categories $\FF_\Lambda R$ and $\oFF_{\Lambda} R,$ and use these to define the categories of fragments and (pre)glider representations.

In section \S\ref{section:StructureResults} we provide a framework in which the category of glider representations fits naturally.  We construct the diagram given in figure \ref{figure:TheDiagram} and show that the top row is a recollement of abelian categories.  The second row is then a restriction of the top row, and the category glider representations occurs naturally via a factorization of the restriction functor $j^*\colon \Preglid FR \to \Prefrag FR.$

In section \S\ref{section:PropertiesOfTheCategoryOfGliderRepresentations} we study the homological properties of the categories in figure \ref{figure:TheDiagram}.  In particular, we show that the categories $\Preglid FR$ and $\Preglid FR$ are Grothendieck quasi-abelian and that $\Glid(FR)$ is a Grothendieck deflation quasi-abelian category.

In sections \S\ref{section:NaturalGliders} and \S\ref{section:Noetherian} we look at some interesting subcategories of the category $\Glid FR$, namely the category $\NGlid FR$ of natural glider representations and the category $\glid FR$ of noetherian glider representations.  We show that both of these still carry the structure of deflation-exact categories.

Section \ref{TheDerivedCategoryOfGliderRepresentationsAsAVerdierLocalization} deals with the (bounded) derived categories of all categories involved.  In particular, we show that the localization sequence $\Mod(R)\to \Preglid FR\to \Glid FR$ induces a Verdier localization sequence $\Db(\Mod(R))\to \Db(\Preglid FR)\to \Db(\Glid FR)$. Moreover, this Verdier localization sequence is equivalent to the Verdier localization sequence $\Db(\Mod(R))\to \Db(\Mod(\oFF_{\Lambda}R))\to \Db(\Mod(\FF_{\Lambda} R))$.

Finally, in section \ref{section:HopfCategories} we show that given a $k$-bialgebra $B$ together with a filtration $FB$ by bialgebras, the companion categories $\FF_{\Lambda}B$ and $\oFF_{\Lambda}B$ are semi-Hopf categories.  In particular, the categories $\Mod_k(\FF_{\Lambda}B)$ and $\Mod_k(\oFF_{\Lambda}B)$ inherit a natural tensor structure. Moreover, the derived equivalences of the previous section are compatible with the tensor structure as well.  We end by showing theorem \ref{theorem:ReconstructionResultIntro} which shows that considering the monoidal category of glider representations of a filtered object retains enough information to reconstruct the original object.

\textbf{Acknowledgements} The authors would like to thank Francesco Genovese, Geoffrey Janssens, Tao Lu, and Wenqing Tao for helpful discussions and comments.  We especially want to thank Frederik Caenepeel and Fred Van Oystaeyen for many useful discussions and ideas.  We are grateful to Geoffrey Janssens and Frederik Caenepeel for sharing an early version of \cite{CaenepeelJanssens19}; section \ref{subsection:Isocategorical} and the results therein were motivated by those discussions and results.  The second author is currently a postdoctoral researcher at FWO (12.M33.16N).

%% file: Preliminaries.tex
\section{Preliminaries}\label{section:Preliminaries}

This section is preliminary in nature.  Throughout, we assume that all rings and ringhomomorphisms are unital.

\subsection{Filtered rings}

Let $\Gamma, \leq$ be an ordered group, i.e.~$\Gamma$ is a group, $\leq$ is an ordering on $\Gamma,$ satisfying the following property:
\[\forall \alpha, \beta, \gamma, \delta \in \Gamma: (\alpha \leq \beta \land \gamma \leq \delta) \Rightarrow \alpha \gamma \leq \beta \delta.\]

\begin{remark}
If $\Gamma, \leq$ is an ordered group, then the poset category of $\Gamma^+ \coloneqq \{\gamma \in \Gamma \mid e \leq \gamma\}$ is filtered.
\end{remark}

Let $R$ be a unital ring.  A \emph{$\Gamma$-filtration} on $R$ is a $\Gamma$-indexed set of additive subgroups $\{F_\gamma R\}_{\gamma \in \Gamma}$ of $R$ satisfying:
\begin{enumerate}
	\item $1_R \in F_e R$,
	\item $\forall \alpha, \beta \in \Gamma: \alpha \leq \beta \Rightarrow F_\alpha R \subseteq F_\beta R$,
	\item $\forall \alpha, \beta \in \Gamma: (F_\alpha R) (F_\beta R) \subseteq F_{\alpha\beta} R$.
\end{enumerate}
A \emph{$\Gamma$-filtered ring} is a ring $R$ together with a $\Gamma$-filtration on $R$.  We write $FR$ for the $\Gamma$-filtered ring $\{F_\gamma R\}_{\gamma \in \Gamma}$.

\begin{remark}
Even though we will assume $\Gamma$ is a group, it is straightforward to generalize our results to the case where $\Gamma$ is a cancellative monoid. 
\end{remark}

\subsection{Fragments and gliders}

The definition of a fragment over a filtered ring has changed since its original definition in \cite{NawalVanOystaeyen95}.  To avoid confusion with the terminology used in \cite{Caenepeel18, CaenepeelVanOystaeyen16, CaenepeelVanOystaeyen18, CaenepeelVanOystaeyenBook19, CaenepeelVanOystaeyen19}, we refer to the objects defined in \cite{NawalVanOystaeyen95} as prefragments.  We start by recalling the definition. 

\begin{definition}\label{definition:Fragment}	
	Let $FR$ be an $\mathbb{N}$-filtered ring with subring $S=F_0R$.  A \emph{(left) prefragment over $FR$} is a left $S$-module together with a descending chain of subgroups \[M=M_0\supseteq M_{-1} \supseteq \dots \supseteq M_{-i}\supseteq \dots\] satisfying the following properties.

	\begin{enumerate}[label=$\mathrm{\mathbf{f_{\arabic*}}}$,start=1]
		\item\label{enumerate:Fragment1} For every $i\in \mathbb{N}$ there is a map $\phi_i\colon F_i R\times M_{-i}\rightarrow M\colon (\lambda, m)\mapsto \lambda\cdot m$ such that 
		\begin{eqnarray*}
			\lambda\cdot (m+n) &=& \lambda\cdot m+\lambda \cdot n,\\
			(\lambda+\mu)\cdot m&=& \lambda\cdot m+\mu \cdot n,\\
			1\cdot m &=& m,
		\end{eqnarray*}
		for all $\lambda,\mu\in F_iR$ and $m,n\in M_i$.
		\item\label{enumerate:Fragment2} For every $i,j \in \mathbb{N}$ with $j+i \leq 0$, there is a commutative diagram 
		\[\xymatrix{
		M & &M_{j-i}\ar@{_{(}->}[ll]
\ar@{^{(}->}[rr]
 && M\\
		F_iR\times M_{-i} \ar[u]_{\phi_i} && F_j R\times M_{-i} \ar@{_{(}->}[ll]^{i_F}
\ar@{^{(}->}[rr]_{i_M}\ar[u]
&& F_jR\times M_{-j} \ar[u]^{\phi_j}
		}\]
		\item \label{enumerate:Fragment3} For every $i,j \in \mathbb{N}$, there is a commutative diagram
		\[\xymatrix{
			F_iR\times F_jR\times M_{-(i+j)}\ar[rr]^{m\times 1_{M_{-(i+j)}}}\ar[d] && F_{i+j}R\times M_{-(i+j)}\ar[d]^{\phi_{i+j}}\\
			F_iR\times M_{-i}\ar[rr]^{\phi_i} && M
		}\]
		Here $m$ denotes the multiplication in $R$ and the left vertical arrow is defined using \ref{enumerate:Fragment2}.
	\end{enumerate}
	
		Let $M$ and $N$ be $FR$-fragments. A \emph{morphism $f\colon M\to N$ of $FR$-fragments} is an $S$-linear map $f\colon M\to N$ such that $f(M_i)\subseteq N_i$ for all $i\in \mathbb{N}$ and $f(r\cdot m)=r\cdot f(m)$ for all $r\in F_iR$ and $m\in M_i$.
\end{definition}

\begin{remark}\label{remark:PrefragmentsAdditiveCategory}
The prefragments over a filtered ring $FR$ form an additive category.
\end{remark}
	
\begin{definition}\label{definition:EarlyDefinitionGliders}
	Let $\{M_i\}_i$ be a prefragment over a filtered ring $FR$.  If the fragmented scalar multiplications $F_i R\times M_{-i}\rightarrow M$ are induced from an $R$-module $\Omega\supseteq M$, we say that $M$ is a \emph{glider representation}.
	
	Let $M$ and $N$ be glider representations.  A morphism $f\colon M\to N$ of prefragments is called a \emph{morphism of glider representations} if there exist $R$-modules $\Omega_M$ and $\Omega_N$ such that $M\subseteq \Omega_M$ and $N\subseteq\Omega_N$ exhibiting that $M,N$ are glider representations and there exists an $R$-linear map $F\colon \Omega_M\to \Omega_N$ such that the following diagram commutes:
	\[\xymatrix{
		\Omega_M\ar[r]^F & \Omega_N\\
		M\ar[r]^f\ar@{^{(}->}[u] & N\ar@{^{(}->}[u]
	}\]
	(thus, the map between the prefragments $M \to N$ needs to be induced by an $R$-module morphism $\Omega_M \to \Omega_N$).
\end{definition}

\begin{remark}\label{remark:Obstruction}
	Despite remark \ref{remark:PrefragmentsAdditiveCategory}, it is not clear that the glider representations form a category. Indeed, let $f\colon A\to B$ and $g\colon B\to C$ be morphisms of $FR$-glider representations.  By definition, there are ambient $R$-modules $\Omega_A,\Omega_B,\Omega_B',\Omega_C$ and $R$-module maps $F\colon \Omega_A\to \Omega_B, G\colon \Omega_B'\to\Omega_C$ such that the diagrams 
\[\xymatrix{
	\Omega_A\ar[r]^F & \Omega_B &&&\Omega_B'\ar[r]^G& \Omega_C\\
	A\ar[r]^f\ar@{^{(}->}[u] & B\ar@{^{(}->}[u] &&& B\ar[r]^g\ar@{^{(}->}[u] & C\ar@{^{(}->}[u]
}\] commute.  However, it is not clear whether the composition $g\circ f$ defines a morphism of $FR$-glider representations.
\end{remark}

\begin{remark}
The definitions of prefragments and glider representations can be adjusted to accommodate more general filtered rings.  For prefragments, this will be done in definition \ref{definition:PrefragAndPreglid}.  In definition \ref{definition:GlidAsLoc}, we will give a different definition for a morphism of gliders.  It will follow from proposition \ref{proposition:PhiFullyFaithful} and \ref{proposition:LeftAdjoint} that this new definition is compatible with the one in definition \ref{definition:EarlyDefinitionGliders}.
\end{remark}

\subsection{Localizations of categories}

We recall some basic results about localizations of categories.  The material of this section is based on \cite{GabrielZisman67}.

\begin{definition}\label{Definition:LocalizationWithRespectToMorphisms}
Let $\CC$ be a small category and let $\Sigma \subseteq \Mor \CC$ be a subset of morphisms of $\CC$.  The \emph{localization of $\CC$ with respect to $\Sigma$} is a functor $Q\colon \CC \to \CC[\Sigma^{-1}]$, universal with respect to the property such that $Q(s)$ is invertible, for all $s \in \Sigma$.
\end{definition}

\begin{remark}
The universality in definition \ref{Definition:LocalizationWithRespectToMorphisms} means that, for every category $\DD$, the functor $(Q \circ -)\colon \Fun(\CC[\Sigma^{-1}], \DD) \to \Fun(\CC, \DD)$ induces an isomorphism between $\Fun(\CC[\Sigma^{-1}], \DD)$ and the full subcategory of $\Fun(\CC, \DD)$ consisting of those functors $F\colon \CC \to \DD$ which make every $s \in \Sigma$ invertible.
\end{remark}

The following proposition is standard (see \cite[proposition I.1.3]{GabrielZisman67}).

\begin{proposition}\label{proposition:GabrielZisman}
Let $F\colon \CC \to \DD$ and $G\colon \DD \to \CC$ be functors between small categories.  We write $\Sigma \subseteq \Mor \CC$ for the set of all morphisms $f$ for which $F(f)$ is invertible.  If $F$ is left adjoint to $G$, then the following are equivalent:
\begin{enumerate}
	\item $G$ is fully faithful,
	\item the counit $FG \Rightarrow 1_\DD$ is invertible,
	\item the unique map $H\colon \CC[\Sigma^{-1}] \to \DD$ making the diagram
	\[\xymatrix{
	{\CC} \ar[rr] \ar[rd]_{Q} && {\DD} \\
	& {\CC[\Sigma^{-1}]} \ar[ru]_H}\]
	commute, is an equivalence.
\end{enumerate}
\end{proposition}

We recall the following proposition from \cite{GabrielZisman67}.

\begin{proposition}\label{proposition:CycleOfAdjoints}
Consider functors $F,G,H$ as in the diagram
\[\xymatrix{
\CC \ar[rr]^F && \DD\ar[dl]^G \\
& \EE \ar[lu]^H}\]
\begin{enumerate}
	\item\label{enumerate:CycleOfAdjointsLoc} Assume that $G\colon \DD \to \EE$ is a localization.
	  \begin{enumerate}
			\item\label{enumerate:CycleOfAdjointsLocInner} If $HG$ is left adjoint to $F$, then $H$ is left adjoint to $GF$.
			\item\label{enumerate:CycleOfAdjointsLocOuter} If $F$ is left adjoint to $HG$, then $GF$ is left adjoint to $H$.
		\end{enumerate}
	\item\label{enumerate:CycleOfAdjointsFF} Assume that $G\colon \DD \to \EE$ is fully faithful.
	\begin{enumerate}
		\item\label{enumerate:CycleOfAdjointsFFInner} If $H$ is left adjoint to $GF$, then $HG$ is left adjoint to $F$
		\item\label{enumerate:CycleOfAdjointsFFOuter} If $GF$ is left adjoint to $H$, then $F$ is left adjoint to $HG$.
	\end{enumerate}
\end{enumerate}
\end{proposition}

\begin{proof}
The statement \eqref{enumerate:CycleOfAdjointsLocInner} is shown in \cite[lemma I.1.3.1]{GabrielZisman67}, and uses that $- \circ G \colon\Fun(\EE, \XX) \to \Fun(\DD, \XX)$ is fully faithful (as $G$ is a localization).  One can prove \eqref{enumerate:CycleOfAdjointsLocOuter} in a similar way.  The other statements can be shown in a similar fashion, using that, for each category $\XX$, the functor $G \circ - \colon\Fun(\XX, \DD) \to \Fun(\XX, \EE)$ is fully faithful.
\end{proof}

In this paper, we often consider localizations with respect to right multiplicative systems.

\begin{definition}\label{definition:RMS}
Let $\CC$ be a category and let $\Sigma \subseteq \Mor \CC$ be a subset of morphisms of $\CC$.  We say that $\Sigma$ is a \emph{right multiplicative system} if the following properties are satisfied.
\begin{enumerate}[label=\textbf{RMS\arabic*},start=1]
	\item\label{RMS1} For every object $A$ of $\CC$ the identity $1_A$ is contained in $\Sigma$, and $\Sigma$ is closed under composition.
	\item\label{RMS2} Every solid diagram
	\[\xymatrix{
		X \ar@{.>}[r]^{g} \ar@{.>}[d]_{t}^{\rotatebox{90}{$\sim$}}& Y\ar[d]_{s}^{\rotatebox{90}{$\sim$}}\\
		Z\ar[r]_{f} & W
	}\] with $s\in \Sigma$ can be completed to a commutative square as above, with $t\in \Sigma$. 
	\item\label{RMS3} For every pair of morphisms $f,g\colon X\rightarrow Y$ and every $s\in \Sigma$ with source $Y$ such that $s\circ f= s\circ g$ there exists a $t\in \Sigma$ with target $X$ such that $f\circ t =g\circ t$.
\end{enumerate}
\end{definition}

For localizations with respect to a right multiplicative system $\Sigma \subseteq \Mor \CC$, we have the following description of the localization $\CC[\Sigma^{-1}]$.

\begin{descript}\label{descript:Localization}
	Let $\CC$ be a category and $\Sigma$ a right multiplicative system in $\CC$.  We define a category $\Sigma^{-1}\CC$ as follows:
	\begin{enumerate}
		\item We have $\Ob(\Sigma^{-1}\CC)=\Ob(\CC)$.
		\item Let $f_1\colon X_1\rightarrow Y, s_1\colon X_1\rightarrow X, f_2\colon X_2\rightarrow Y, s_2\colon X_2\rightarrow X$ be morphisms in $\CC$ with $s_1,s_2\in \Sigma$.  We say that the pairs $(f_1,s_1), (f_2,s_2) \in (\Mor \CC) \times \Sigma$ are equivalent (denoted by $(f_1,s_1) \sim (f_2,s_2)$) if there exists a third pair $(f_3\colon X_3\rightarrow Y,s_3\colon X_3\rightarrow X) \in (\Mor \CC) \times \Sigma$ and morphisms $u\colon X_3\rightarrow X_1, v\colon X_3\rightarrow X_2$ such that 
		\[\xymatrix@!{
			& X_1\ar[ld]_{s_1}^{}\ar[rd]^{f_1} & \\
			X &X_3\ar[d]^{v}\ar[u]_{u}\ar[l]_{s_3}^{}\ar[r]_{f_3} & Y\\
			& X_2 \ar[ul]^{s_2}_{}\ar[ur]_{f_2}&
		}\] is a commutative diagram.
		\item $\Hom_{\Sigma^{-1}\CC}(X,Y)=\left\{(f,s)\mid f\in \Hom_{\CC}(X',Y), s\colon X'\rightarrow X \mbox{ with } s\in \Sigma \right\} / \sim$
		\item The composition of $(f\colon X'\rightarrow Y, s\colon X'\rightarrow X)$ and $(g\colon Y'\rightarrow Z, t\colon Y'\rightarrow Y)$ is given by $(g\circ h\colon X''\rightarrow Z,s\circ u\colon X''\rightarrow X)$ where $h$ and $u$ are chosen to fit in a commutative diagram 
		\[\xymatrix{
		X''\ar[r]^{h}\ar[d]_{u}^{\rotatebox{90}{$\sim$}} & Y'\ar[d]_{t}^{\rotatebox{90}{$\sim$}}\\
		X'\ar[r]^{f} & Y
		}\] which exists by \ref{RMS2}.
	\end{enumerate}                                                                                                       
\end{descript}

The canonical functor $Q\colon \CC \to \Sigma^{-1}\CC$ satisfies the conditions of definition \ref{Definition:LocalizationWithRespectToMorphisms}. 

\subsection{Representations of small preadditive categories}\label{subsection:RepresentationsOfCategories}

Let $k$ be a commutative ring.  Let $\aa$ be a small $k$-linear category.  A \emph{(left) $\aa$-module} is a covariant $k$-linear functor from $\aa$ to Mod $k$, the category of all $k$-modules.  The category of all left $\aa$-modules is denoted by $\Mod \aa$.  If we do not specify the ring $k$, we will take $k=\mathbb{Z}.$

It follows from the Yoneda lemma that, for every $A \in \aa$, the representable $\aa$-module $\aa(A,-)$ is projective.  We refer to such an $\aa$-module as a \emph{standard projective $\aa$-module}.  It is clear that every finitely generated projective is a direct summand of a finite direct sum of standard projectives.  If $\aa$ has finite direct sums and idempotents split in $\aa$, then every finitely generated projective is isomorphic to a standard projective.

If $f\colon \frak{a} \to \frak{b}$ is a functor between small preadditive categories, then there is an obvious restriction functor
\[(-)_{\frak{a}}\colon \Mod(\frak{b})\to \Mod(\frak{a})\]
which sends $N$ to $N \circ f$.  This restriction functor has a left adjoint 
\[\bb \otimes_{\aa} -\colon \Mod(\frak{a})\to \Mod(\frak{b})\]
which is the right exact functor sending the projective generators
$\frak{a}(A,-)$ in $\Mod(\frak{a})$ to $\frak{b}(f(A),-)$ in $\Mod(\frak{b})$.  If $f\colon \aa \to \bb$ is fully faithful, then the natural morphism $M \to (\bb \otimes_\aa M)_\aa$ is an isomorphism, for all $M \in \Mod \aa.$  It follows from \cite[Chapter I.3]{GabrielZisman67} that $\Mod \aa \simeq (\Mod \bb) / \ker (-)_\aa $.

\subsection{Recollements}

Recollements were introduced in a triangulated context in \cite{BeilinsonBernsteinDeligne82}.  To establish notations, we recall the definition of a recollement in both an abelian and a triangulated setting.

\begin{definition}
	Let $\AA$ and $\BB$ be categories and let $L,R\colon \BB\to \AA$ and $E\colon \AA\to \BB$ be functors.  The triple $(L,E,R)$ is called an \emph{adjoint triple} if $L \dashv E$ and $E \dashv R$.
\end{definition}

\begin{definition}\label{definition:AbelianRecollement}
	A \emph{recollement of abelian categories} is a triple of abelian categories $\AA,\BB$ and $\CC$ and a diagram
	\[\xymatrix{
		\AA\ar[rr]|{i_*} && \BB \ar@/^/[ll]^{i^!}\ar@/_/[ll]_{i^*} \ar[rr]|{j^*} && \CC\ar@/^/[ll]^{j_*}\ar@/_/[ll]_{j_!}
	}\] with $6$ additive functors satisfying the following conditions:
	\begin{enumerate}
		\item the triple $(i^*,i_*,i^!)$ is an adjoint triple,
		\item	the triple $(j_!,j^*,j_*)$ is an adjoint triple,
		\item the functors $i_*,j_*$ and $j_!$ are fully faithful,
		\item $\im(i_*)=\ker(j^*)$.
	\end{enumerate}
\end{definition}

\begin{remark}\label{remark:Recollements}
\begin{enumerate}
  \item It follows from proposition \ref{proposition:GabrielZisman} that $\CC \simeq \Sigma^{-1}\BB$ where $\Sigma \subseteq \Mor \BB$ is the class of morphisms that become invertible under $j^*$.  These are the morphisms with kernel and cokernel in $i_*(\AA) \subseteq \BB.$
	\item By proposition \ref{proposition:GabrielZisman}, the conditions in definition \ref{definition:AbelianRecollement} are not minimal; it suffices to determine the adjoint triple $(j_!,j^*,j_*)$ where either $j_!$ or $j_*$ is fully faithful (see, for example, \cite[remark 2.3]{Psaroudakis14}).
	\item We will be interested in recollements of abelian categories where all three categories are module categories.  It is shown in \cite{PsaroudakisVitoria14} that these recollements are classified by an idempotent (see also \cite{Jans65}).
\end{enumerate}
\end{remark}

In the same fashion, one defines a recollement of triangulated categories.

\begin{definition}\label{definition:TriangulatedRecollement}
	A \emph{recollement of triangulated categories} is a triple of triangulated categories $\TT'',\TT$ and $\TT'$ and a diagram
	\[\xymatrix{
		\TT''\ar[rr]|{i_*} && \TT \ar@/^/[ll]^{i^!}\ar@/_/[ll]_{i^*} \ar[rr]|{j^*} && \TT'\ar@/^/[ll]^{j_*}\ar@/_/[ll]_{j_!}
	}\] with $6$ triangulated functors satisfying the following conditions:
	\begin{enumerate}
		\item the triple $(i^*,i_*,i^!)$ is an adjoint triple,
		\item	the triple $(j_!,j^*,j_*)$ is an adjoint triple,
		\item the functors $i_*,j_*$ and $j_!$ are fully faithful,
		\item $j^*i_*=0$,
		\item for each $X\in \TT$, there are two triangles in $\TT$:
		\[i_*i^!(X)\to X \to j_*j^*(X)\to \Sigma i_*i^!(X) \mbox{ and } j_!j^*(X)\to X \to i_*i^*(X)\to \Sigma j_!j^*(X).\]
	\end{enumerate}
\end{definition}         

\begin{remark}
The conditions in definition \ref{definition:TriangulatedRecollement} are not minimal; it suffices to determine the adjoint triple $(j_!,j^*,j_*)$ of exact functors where either $j_!$ or $j_*$ is fully faithful (see, for example, \cite[proposition 1.14]{Heider07}).
\end{remark}

\subsection{One-sided exact categories and admissibly percolating subcategories}\label{subsection:LocalizationsPreliminaries}

One-sided exact categories were introduced in \cite{BazzoniCrivei13, Rump01} as a framework for studying one-sided quasi-abelian categories \cite{Rump01}.  We recall some definitions as well as some results concerning quotients of one-sided exact categories by percolating subcategories (\cite{HenrardVanRoosmalen19b, HenrardVanRoosmalen19a}).

\begin{definition}\label{definition:ConflationCategory}
	Let $\CC$ be an additive category.  We say that a sequence $A\xrightarrow{f} B\xrightarrow{g} C$ is a \emph{kernel-cokernel pair} if $f = \ker g$ and $g = \coker f$.
	A \emph{conflation category} $\CC$ is an additive category with a chosen class of kernel-cokernel pairs, called \emph{conflations}, closed under isomorphisms.  Given a conflation $A\xrightarrow{f} B\xrightarrow{g} C$, the map $f$ is called an \emph{inflation} and the map $g$ is called a \emph{deflation}.  Inflations are denoted by $\rightarrowtail$ and deflations by $\twoheadrightarrow$.  A map $f\colon X\rightarrow Y$ is called \emph{admissible} if it admits a deflation-inflation factorization, i.e.~$f$ factors as $X\twoheadrightarrow Z\rightarrowtail Y$.
		
	Let $\CC$ and $\DD$ be conflation categories.  An additive functor $F\colon \CC\rightarrow \DD$ is called \emph{exact} or \emph{conflation-exact} if it maps conflations in $\CC$ to conflations in $\DD.$
\end{definition}

\begin{definition}\label{definition:RightExact}
	A conflation category $\CC$ is called \emph{right exact} or \emph{deflation-exact} if it satisfies the following axioms:
	\begin{enumerate}[label=\textbf{R\arabic*},start=0]
		\item\label{R0} The identity morphism $1_0\colon 0\rightarrow 0$ is a deflation.
		\item\label{R1} The composition of two deflations is a deflation.
		\item\label{R2} Pullbacks along deflations exist and deflations are stable under pullbacks.
	\end{enumerate}
	
	 Dually, we call a conflation category $\CC$ \emph{left exact} or \emph{inflation-exact} if the opposite category $\CC^{op}$ is right exact.  For completeness, an inflation-exact category is a conflation category such that the distinguished class of conflations satisfies the following axioms:
		\begin{enumerate}[label=\textbf{L\arabic*},start=0]
		\item\label{L0} The identity morphism $1_0\colon 0\rightarrow 0$ is an inflation.
		\item\label{L1} The composition of two inflations is an inflation.
		\item\label{L2} Pushouts along inflations exist and inflations are stable under pushouts.
	\end{enumerate}
	
	 A conflation category which is both inflation-exact and deflation-exact is a (Quillen) \emph{exact category}.
\end{definition}

\begin{definition}\label{definition:DeflationQuasiAbelian}
	A category $\CC$ is called \emph{left quasi-abelian} or \emph{deflation quasi-abelian} if it is a pre-abelian category, i.e.~an additive category with kernels and cokernels, such that cokernels are stable under pullbacks. 	Dually, $\CC$ is called \emph{right quasi-abelian} or \emph{inflation quasi-abelian} if it is a pre-abelian category such that kernels are stable under pushouts.
	
	The category $\CC$ is called \emph{quasi-abelian} if it is both left and right quasi-abelian.
\end{definition}

\begin{remark}\label{remark:QuasiAb}
	\begin{enumerate}
		\item A deflation quasi-abelian category can be given the structure of a deflation-exact category by choosing all kernel-cokernel pairs as the conflations.  Dually, an inflation quasi-abelian category can be endowed with the structure of an inflation-exact category by choosing all kernel-cokernel pairs as conflations.  To avoid confusion, we prefer using ``inflation-exact'' and ``deflation-exact'' over ``left exact'' and ``right exact''.
		\item A quasi-abelian category has a natural exact structure (the conflations are given by all kernel-cokernel pairs).
	\end{enumerate}
\end{remark}

\begin{definition}\label{definition:GeneralPercolatingSubcategory}
	Let $\CC$ be a deflation-exact category. A non-empty full subcategory $\AA$ of $\CC$ satisfying the following three axioms is called an \emph{admissibly deflation-percolating subcategory} of $\CC$.
	\begin{enumerate}[label=\textbf{A\arabic*},start=1]
		\item\label{A1} $\AA$ is a \emph{Serre subcategory}, that is:
		\[\mbox{ If } A'\rightarrowtail A \twoheadrightarrow A'' \mbox{ is a conflation in $\CC$, then } A\in \Ob(\AA) \mbox{ if and only if } A',A''\in \Ob(\AA).\]
		\item\label{A2} For all morphisms $C\rightarrow A$ with $C \in \Ob(\CC)$ and $A\in \Ob(\AA)$, there exists a commutative diagram
		\[\xymatrix{
		A'\ar@{>->}[rd] & \\
		C \ar@{->>}[u]\ar[r]& A\\
				}\] with $A'\in \Ob(\AA)$.
		\item\label{A3} If $a\colon C\rightarrowtail D$ is an inflation and $b\colon C\twoheadrightarrow A$ is a deflation with $A\in \Ob(\AA)$, then the pushout of $a$ along $b$ exists and yields an inflation and a deflation, respectively, i.e.
		\[\xymatrix{
			C \ar@{>->}[r]^{a}\ar@{->>}[d]^b & D\ar@{.>>}[d]\\
			A\ar@{>.>}[r] & P
		}\]
	\end{enumerate}
\end{definition}

\begin{remark}\label{remark:PercolatingRemark}
	\begin{enumerate}
		\item An admissibly deflation-percolating subcategory $\AA$ of a deflation-exact category $\CC$ is automatically an abelian category (see \cite[proposition~6.4]{HenrardVanRoosmalen19a}). 
		\item When $\EE$ is an exact category, any full subcategory $\AA\subseteq \EE$ satisfies axiom \ref{A3} (this follows, for example, from \cite[proposition~2.12]{Buhler10}).
	\end{enumerate}
\end{remark}

\begin{definition}\label{definition:WeakIsomorphism}
	Let $\AA$ be an admissibly deflation-percolating subcategory of a deflation-exact category $\CC$. A morphism $f$ is called a \emph{weak $\AA$-isomorphism} if it is admissible and $\ker(f),\coker(f)\in \AA.$	 A weak isomorphism will often be endowed with $\sim$, i.e.~we write $f\colon A \stackrel{\sim}{\rightarrow} B$ for a weak isomorphism.
\end{definition}

\begin{remark}\label{remark:WeakIsomorphismsAdmissible}
A morphism $f\colon X \to Y$ is a weak $\AA$-isomorphism if it has a kernel and a cokernel (both lying in $\AA$) and factors as $X \deflation X/(\ker f) \inflation Y.$
\end{remark}

The following theorem summarizes various results from \cite{HenrardVanRoosmalen19a}.

\begin{theorem}\label{theorem:MainTheoremPercoLocalization}
	Let $\AA$ be an admissibly deflation-percolating subcategory of a delation-exact category $\CC$.
	\begin{enumerate}
		\item The set $\Sigma_{\AA}$ of weak $\AA$-isomorphisms is a right multiplicative system.
		\item  The weakest conflation structure on $\Sigma_{\AA}^{-1}\CC$ for which $Q\colon \CC \to \Sigma^{-1}_\AA \CC$ is conflation-exact, makes $\Sigma^{-1}_\AA \CC$ into a deflation-exact category.
		\item The localization functor $Q$ is a conflation-exact functor and is 2-universal among conflation-exact functors $F\colon \CC\to \DD$ such that $F(\AA)=0$.  Here, $\DD$ is any deflation-exact category.
		\item The right multiplicative system $\Sigma_{\AA}$ is saturated, i.e.~$Q(f)$ is an isomorphism if and only if $f\in \Sigma_{\AA}$.
		\item Pullbacks alongs weak isomorphisms exist and weak isomorphisms are stable under pullbacks.
	\end{enumerate}
\end{theorem}

\begin{remark}
Because of the above universal property, we often write $\CC / \AA$ for the category $\Sigma_{\AA}^{-1}\CC.$
\end{remark}

\subsection{Torsion and torsion-free classes}\label{subsection:Torsion}

Let $\AA$ be an abelian category.  A \emph{torsion theory} on $\AA$ is a pair $(\TT, \FF)$ of full subcategories of $\AA$ so that $\Hom_\AA(\TT,\FF) = 0$ and for every object $A \in \AA$ there is a short exact sequence
\[0 \to T \to A \to F \to 0\]
with $T \in \TT$ and $F \in \FF$.  This short exact sequence is necessarily unique up to isomorphism.  The subcategory $\TT$ is called the \emph{torsion subcategory} and the category $\FF$ is called a \emph{torsion-free subcategory}.

Any full subcategory of $\AA$ satisfying the following properties is a torsion subcategory:
\begin{enumerate}
\item $\FF$ is closed under extensions and subobjects,
\item the embedding $\FF \to \AA$ has a left adjoint.
\end{enumerate}
The associated torsion subcategory is then given by
\[\TT = {}^{\perp_0}\FF = \{A \in \AA \mid \Hom(A, \FF) = 0\}.\]

Let $(\TT, \FF)$ be a torsion theory.  We say that $(\TT, \FF)$ is \emph{tilting} if for every $A \in \AA$ there is a monomorphism $A \to T$ for some $T \in \TT$.  Likewise, we say that $(\TT, \FF)$ is \emph{cotilting} if for every $A \in \AA$ there is an epimorphism $F \to A$ for some $F \in \AA$.

\begin{remark}
If $\AA$ has enough projectives, then a torsion theory $(\TT, \FF)$ is cotilting if and only if all projective objects lie in $\FF.$
\end{remark}

We recall the following result from \cite[proposition B.3]{BondalVandenBergh03}.

\begin{theorem}\label{theorem:BondalVandenBergh}
Let $\CC$ be an additive category.  The following are equivalent.
\begin{enumerate}
	\item $\CC$ is a quasi-abelian category,
	\item There is a cotilting torsion theory $(\TT , \FF)$ in an abelian category $\AA$ with $\CC \simeq \FF.$
	\item There is a tilting torsion theory $(\TT , \FF)$ in an abelian category $\AA$ with $\CC \simeq \TT.$
\end{enumerate}
\end{theorem}

\subsection{The derived category of a deflation-exact category}

We recall the definition of the derived category of a deflation-exact category from \cite{BazzoniCrivei13, HenrardVanRoosmalen19b}. The definition is similar to the derived category of an exact category \cite{Buhler10, Neeman01}.  For an additive category $\EE$, we write $\C(\EE)$ for the category of complexes and $\K(\EE)$ for the homotopy category.  We write $\Cb(\EE)$ and $\Kb(\EE)$ for the bounded variants.

\begin{definition}\label{definition:AcylicComplex}
	Let $\EE$ be a deflation-exact category. A complex $X^{\bullet} \in \Cb(\EE)$ is called \emph{acylic in degree $n$} if $d_X^{n-1}$ factors as 
	\[\xymatrix{
	X^{n-1}\ar[rr]^{d_X^{n-1}}\ar@{->>}[rd]^{p^{n-1}} && X^n\\
	 & \ker(d_X^{n})\ar@{>->}[ru]^{i^{n-1}} &
	}\] where the deflation $p^{n-1}$ is the cokernel of $d_X^{n-2}$ and the inflation $i^{n-1}$ is the kernel of $d_X^{n}$.
	
	A complex $X^{\bullet}$ is called \emph{acyclic} if it is acylic in each degree.  The full subcategory of $\Cb(\EE)$ of acyclic complexes is denoted by $\Acb_{\C}(\EE)$. We write $\Acb_{\K}(\EE)$ for the full subcategory of acyclic complexes when viewed as a subcategory of $\Kb(\EE)$. We simply write $\Acb(\EE)$ if there is no confusion.
\end{definition}

In \cite{BazzoniCrivei13}, it is shown that $\Ac_{\K}(\EE)$ is a triangulated subcategory (not necessarily closed under isomorphisms) of $\K(\EE)$.

\begin{definition}\label{definition:DerivedCategory}
	Let $\EE$ be a deflation-exact category.  The bounded derived category $\Db(\EE)$ is defined as the Verdier localization $\Kb(\EE)/\Acb_{\K}(\EE)$.
\end{definition}

The derived category $\Db(\EE)$ enjoys many standard properties as in the exact case. We refer the reader to \cite{HenrardVanRoosmalen19b} for details and precise statements. The following theorem is \cite[theorem~1.4]{HenrardVanRoosmalen19b}.

\begin{theorem}\label{theorem:MainTheoremDerivedCatOfLoc}
	Let $\EE$ be a deflation-exact category and let $\AA\subseteq \EE$ be an admissibly deflation-percolating subcategory. The following sequence is a Verdier localization sequence
	\[\DAb(\EE)\to \Db(\EE)\to \Db(\EE/\AA).\]
	Here, $\DAb(\EE)$ is the thick triangulated subcategory of $\Db(\EE)$ generated by $\AA\subseteq \Db(\EE)$.
\end{theorem}

\subsection{Semi-Hopf categories}\label{susection:SemiHopf}

Let $\VV$ be a strict braided monoidal category and let $\underline{\CC}(\VV)$ be the category of coalgebra (or comonoid) objects in $\VV$ with coalgebra morphisms. Note that $\underline{\CC}(\VV)$ is itself a monoidal category and the unit object $k$ of $\VV$ is a coalgebra.

\begin{definition}
Let $\VV$ be a strict braided monoidal category. A category $\CC$ enriched over $\underline{\CC}(\VV)$ is called a \emph{semi-Hopf category}.	If $\VV=\operatorname{Vec}_k$ the category of vector spaces over a field $k$, a category $\CC$ enriched over $\underline{\CC}(\operatorname{Vec}_k)$ is called a \emph{$k$-linear semi-Hopf category}.
\end{definition}

When $\CC$ is a $k$-linear semi-Hopf category, the category $\Mod_k \CC = \Fun_k(\CC, \operatorname{Vec}_k)$ of $k$-linear $\CC$-modules has an induced pointwise monoidal structure (see \cite[proposition 3.2]{BatistaCaenepeelVercruysse16} for details).

\begin{remark}
The terminology of a semi-Hopf category has been introduced in \cite{BucklyFieremansVasilakopoulouVercruysse17, CaenepeelFieremans18}.  A semi-Hopf category with an antipode is called a \emph{Hopf category}, and was earlier introduced in \cite{BatistaCaenepeelVercruysse16}.
\end{remark}

%% file: CategoryGlider.tex
\section{The category of glider representations}\label{section:GliderCategory}

Let $\Gamma$ be an ordered group and let $\Lambda \subseteq \Gamma$ be any subset.  Let $FR$ be a $\Gamma$-filtered ring.  In this section, we collect all definitions necessary to define the category $\Glid FR$ of $\Lambda$-glider representations over a filtered ring $FR$.  We refer to the diagram in figure \ref{figure:TheDiagram} for an overview.  We start by defining the categories $\FF_\Lambda R$ and $\oFF_\Lambda R$, and proceed by defining $\Prefrag FR$ and $\Preglid FR$ as full subcategories of $\Mod \FF_\Lambda R$ and $\Mod \oFF_\Lambda R$, respectively.  The category $\Glid FR$ of glider representation is then defined as a localization of the category $\Preglid FR$ of pregliders (see definition \ref{definition:GlidAsLoc}).

The categories $\Mod \FF_\Lambda R$ and $\Prefrag FR$ do not occur in the definition of glider representations, but will occur in subsequent sections.

\subsection{Companion categories over a filtered ring}

Let $(\Gamma, \leq)$ be an ordered group and let $FR$ be a $\Gamma$-filtered ring.

\begin{definition}\label{definition:CompanionCategory}
Let $FR$ be a $\Gamma$-filtered ring and let $\Lambda \subseteq \Gamma$ be any subset.
\begin{enumerate}
	\item We define the \emph{$\Lambda$-filtered companion category} $\FF_\Lambda R$ of $FR$ as follows.  The objects are given by $\Ob (\FF_\Lambda R) = \Lambda$; the morphisms are given by 
		\[\Hom_{\FF_{\Lambda}R}(\alpha, \beta) = \begin{cases} 
	F_{\beta\alpha^{-1}} R & \alpha \leq \beta, \\ 
	0 & \mbox{otherwise.}
	\end{cases}\]
		The composition is given by the multiplication in $R.$
	\item We define the \emph{extended $\Lambda$-filtered companion category} $\overline{\FF}_\Lambda R$ of $FR$ as follows.  The objects are given by $\Ob (\oFF_\Lambda R) = \Lambda \coprod \{\infty\}$; the morphisms are given by
	\[\Hom_{\oFF_{\Lambda}R}(\alpha, \beta) = \begin{cases} 
	\Hom_{\FF_{\Lambda} R}(\alpha, \beta) & \alpha, \beta \in \Lambda, \\ 
	R & \beta = \infty, \\
	0 & \mbox{otherwise.}
	\end{cases}\]
		The composition is given by the multiplication in $R.$
\end{enumerate}
We write $j\colon \FF_\Lambda R \to \oFF_\Lambda R$ for the inclusion functor.
\end{definition}

\begin{notation}
For each $\alpha, \beta \in \Ob (\oFF_\Lambda)$ such that $\alpha \leq \beta$ (or $\beta = \infty$), there is an element $1_R \in \Hom(\alpha,\beta)$.  We refer to this element as $1_{\alpha, \beta}$.
\end{notation}

\subsection{Pregliders and prefragments}\label{subsection:PreglidPrefrag}

Having introduced the categories $\FF_\Lambda R$ and $\oFF_\Lambda R$, we can now define the categories $\Preglid FR$ and $\Prefrag FR.$

\begin{definition}\label{definition:PrefragAndPreglid}
Let $FR$ be a $\Gamma$-filtered ring and let $\Lambda \subseteq \Gamma$ be any subset.  With definitions as above, we have:
\begin{enumerate}
	\item The category $\Prefrag FR$ of \emph{$FR$-prefragments} is the full additive subcategory of $\Mod (\FF_\Lambda R)$ given by those $M \in \Mod (\FF_\Lambda R)$ which satisfy:
		\[\mbox{for all $\alpha \leq \beta$ in $\Ob (\FF_\Lambda R)$, the map $M(1_{\alpha, \beta})\colon M(\alpha) \to M(\beta)$ is a monomorphism.}\]
	We write $\eta\colon \Prefrag FR\hookrightarrow \Mod(\FF_{\Lambda} R)$ for the inclusion $\Prefrag FR\subseteq \Mod(\FF_\Lambda R).$
	\item The category $\Preglid FR$ of \emph{$FR$-pregliders} is the full additive subcategory of $\Mod(\oFF_\Lambda R)$ given by those $M \in \Mod(\oFF_\Lambda R)$ which satisfy:
	\[\mbox{for all $\alpha \leq \beta$ in $\Ob (\oFF_\Lambda R)$, the map $M(1_{\alpha, \beta})\colon M(\alpha) \to M(\beta)$ is a monomorphism.}\]
	We write $\iota\colon \Preglid FR\hookrightarrow \Mod(\oFF_\Lambda R)$ for the inclusion $\Preglid FR\subseteq  \Mod(\oFF_\Lambda R)$.
\end{enumerate}
\end{definition}

\begin{remark}\label{remark:SufficientToCheckMapsToOmega}
For an object $M\in \Ob (\Mod \oFF_\Lambda R)$ to be a preglider, it suffices to check that, for each $\alpha \in \Lambda$, the map $M(1_{\alpha, \infty})\colon M(\alpha) \to M(\infty)$ is a monomorphism.
\end{remark}

\begin{example}\label{example:WhenSchapiraSchneiders16}
When $\Lambda = \Gamma$ and $FR$ is a $\Gamma^+$-filtered ring, then $\Prefrag FR$ is equivalent to the category of filtered $FR$-modules in the sense of \cite{SchapiraSchneiders16}.  In particular, when $\Gamma = \bZ$, we recover the usual notion of a $\bZ$-filtered module (see, for example, \cite{NastasescuVanOystaeyen79}).
\end{example}

\begin{remark}
Let $\Gamma = \bZ$ and let $\Lambda = \bZ^{\leq 0}.$  Let $FR$ be a $\bZ$-filtered ring.  The category $\Prefrag FR$ is equivalent to the category of prefragments from definition \ref{definition:Fragment}.  Indeed, given a prefragment
\[M=M_0\supseteq M_{-1} \supseteq \dots \supseteq M_{-i}\supseteq \dots\]
in the sense of definition \ref{definition:Fragment}, we find a functor $\FF_\Lambda R \to \Ab$ by $-i \mapsto M_{-i}.$  The fragmented action $F_i R \times M_{-i} \to M_0$ is given by the action of $\Hom_{\FF R}(-i,0) = F_i R$.  In particular, the action of $1_{-i,0}$ on $M_{-i}$ is given by the inclusion $M_{-i} \subseteq M_0.$
\end{remark}

\begin{example}\label{example:StandardProjectives}
Let $\Gamma$ be a partially ordered group and $\Lambda \subseteq \Gamma$.  Let $FR$ be a $\Gamma$-filtered ring.  For each $\lambda \in \Lambda$, the representable functor $\FF_\Lambda (\lambda,-)$ is a prefragment and the representable functor $\oFF_\Lambda (\lambda,-)$ is a preglider.  Indeed, the conditions in definition \ref{definition:PrefragAndPreglid} are easy to verify in this case.
\end{example}

\subsection{The category of glider representations}

Let $(\Gamma, \leq)$ be an ordered group and let $FR$ be a $\Gamma$-filtered ring.  Let $\Lambda \subseteq \Gamma$ be any subset.  Consider the set $\Sigma \subseteq \Mor(\Preglid FR)$ given by:
\[(f\colon N \to M) \in \Sigma \quad \Leftrightarrow \quad \mbox{for all $\lambda \in \Lambda,$ the map $f_\lambda\colon N(\lambda) \to M(\Lambda)$ is an isomorphism.}\]
Note that we do not require $f_\infty\colon M(\infty) \to N(\infty)$ to be an isomorphism.

\begin{definition}\label{definition:GlidAsLoc}
	The category $\Glid FR$ is defined as the localization of the category $\Preglid FR$ with respect to $\Sigma$, i.e.~$\Glid FR \coloneqq (\Preglid FR)[\Sigma^{-1} ]$.
\end{definition}

\begin{notation}
 We denote the localization functor by $Q\colon \Preglid FR\to \Glid FR$.
\end{notation}

We will now work to understand this localization better: we will show that $\Sigma$ is a saturated right multiplicative system, so that morphisms in $\Glid FR$ are described by ``right roofs'' as in description \ref{descript:Localization}.  We start with the following lemma.

\begin{lemma}\label{lemma:PhiFullyFaithful}
Let $f\colon M \to N$ be a morphism in $\Mod \oFF_\Lambda R$.  Assume that $\ker f_\lambda = 0$, for all $\lambda \in \Lambda$.  If $N \in \Preglid FR$, then $M \in \Preglid FR.$
\end{lemma}

\begin{proof}
As in remark \ref{remark:SufficientToCheckMapsToOmega}, it suffices to show that, for each $\lambda \in \Lambda$, the map $M(1_{\lambda, \infty})\colon M(\lambda) \to M(\infty)$ is a monomorphism.  Consider the commutative diagram
\[
\xymatrix{
M(\lambda) \ar[rr]^{M(1_{\lambda, \infty})} \ar[d]^{f_\lambda} && M(\infty) \ar[d]^{f_\infty} \\
N(\lambda) \ar[rr]^{N(1_{\lambda, \infty})} && N(\infty)}
\]
given by the morphism $f\colon M \to N$.  As $f_\lambda$ is a monomorphism by assumption and $N$ is a preglider, we know that the left-lower branch composes to a monomorphism.  We now see that $M(1_{\lambda, \infty})\colon M(\lambda) \to M(\infty)$ is a monomorphism as well.
\end{proof}

\begin{remark}
Note that every $s \in \Sigma$ satisfies the conditions in lemma \ref{lemma:PhiFullyFaithful}.
\end{remark}

\begin{proposition}\label{proposition:SigmaRMS}
The set $\Sigma \subset \Mor (\Preglid FR)$ is a saturated right multiplicative system.
\end{proposition}

\begin{proof}
We let $\Theta = \{f \in \Mor (\Mod \oFF_\Lambda R) \mid \forall \lambda \in \Lambda: \ker f_\lambda = 0 = \coker f_\lambda\},$ thus a morphism $f\colon M \to N$ is in $\Theta$ if and only if the restriction $(f)_{\FF_\Lambda R} = f \circ j\colon M \circ j \to N \circ j$ is invertible.  We know that $(\Mod \oFF_\Lambda R)[\Theta^{-1}] \cong \Mod \FF_\Lambda R$ and that $\Theta$ is a saturated right (as well as a left) multiplicative system in $\Mod \oFF_\Lambda R$.

Note that $\Sigma = \Theta \cap \Mor (\Preglid FR).$  Using lemma \ref{lemma:PhiFullyFaithful}, it is straightforward to show that $\Sigma$ is a saturated right multiplicative system.
\end{proof}

\begin{corollary}
The category $\Glid FR$ is an additive category.
\end{corollary}

\begin{proof}
This follows directly from proposition \ref{proposition:SigmaRMS} (see \cite[corollary I.3.3]{GabrielZisman67}).
\end{proof}

\begin{remark}
In \S\ref{subsection:GlidAsQuotient}, we obtain a different proof of proposition \ref{proposition:SigmaRMS} by interpreting the category of glider representations as the quotient $\Preglid FR / i_*(\Mod R).$
\end{remark}

\begin{example}\label{example:UsualFiltrations}
When $\Gamma = \Lambda = \bZ$, then $\Prefrag FR$ is the usual category of filtered modules over the filtered ring $FR.$  The objects in $\Prefrag FR$ are given by $\bZ$-filtered modules $\{M_i\}_{i \in \bZ}$; the objects in $\Preglid FR$ are given by $\bZ$-filtered modules $\{M_i\}_{i \in \bZ}$, together with an $R$-module $M$ and a monomorphism $\varinjlim_{i \in \bZ} M_i \to M.$  The category $\Glid FR$ of glider representations is equivalent to the category $\Prefrag FR$ of prefragments, which is itself equivalent to the usual category of $\bZ$-filtered $FR$-modules.
\end{example}

\begin{example}\label{example:EasyGliders}
Let $\Gamma = \bZ$ and $\Lambda = \{0\}.$  Let $R = k[t]$ for a commutative ring $k$, filtered in the usual way (with $\deg t = 1$).  We consider the following pregliders $M,N \in \Preglid FR$:
\begin{align*}
M(i) &= \begin{cases} k & i=0 \\ k[t]/(t) & i=\infty \end{cases} &
N(i) &= \begin{cases} k & i=0 \\ k[t]/(t-1) & i=\infty \end{cases}
\end{align*}
Note that $\Hom_{\Preglid FR}(M,N) = 0$ while $\Hom_{\Prefrag}(M \circ j, N \circ j) \cong k$ (here, $j\colon \FF_\Lambda R \to \oFF_\Lambda R$ is the inclusion).  Indeed, as $\Lambda = \{0\}$, we find that $\Prefrag FR \simeq \Mod k$ and $M \circ j \cong N \circ j $ are simple in $\Prefrag FR$.  In $\Glid FR$, an isomorphism $Q(M) \to Q(N)$ is given by a roof:
\[\xymatrix{
M && k \ar[r] & k[t]/(t) \\
P \ar[u]^{\rotatebox{90}{$\sim$}} \ar[d]_{\rotatebox{90}{$\sim$}} && k \ar[r] \ar@{=}[u] \ar@{=}[d] & k[t] \ar[u] \ar[d] \\
N && k \ar[r] & k[t]/(t-1)}\]
\end{example}

%% file: StructureResults.tex
\section{Glider representations as prefragments and pregliders}\label{section:StructureResults}

Let $\Gamma$ be an ordered group and let $\Lambda \subseteq \Gamma$ be any subset.  Let $FR$ be a $\Gamma$-filtered ring.  Having defined the categories in figure \ref{figure:TheDiagram}, we now turn our attention to the remaining functors in the same diagram.  Our main goals are to introduce the (fully faithful) functors $L$ and $\phi$, allowing us to interpret the category of glider representations as full subcategories of $\Preglid FR$ and $\Prefrag FR$, respectively (see propositions \ref{proposition:PhiHasLeftAdjoint} and \ref{proposition:LeftAdjoint} below).

We start with completing the localization sequence in the top row of figure \ref{figure:TheDiagram} to a recollement of abelian categories (proposition \ref{proposition:AbelianRecollement} below).

\subsection{A recollement}\label{subsection:recollement}

We first consider the top row of figure \ref{figure:TheDiagram}:
\[\xymatrix{
{\Mod R} \ar[rr]|-{i_*} && {\Mod \oFF_\Lambda R} \ar[rr]|-{j^*} \ar@/^/[ll]^-{i^!} \ar@/_/[ll]_-{i^*} && {\Mod \FF_\Lambda R} \ar@/^/[ll]^-{j_*} \ar@/_/[ll]_-{j_!}
}\]

\begin{itemize}
	\item The functor $j^*\colon \Mod(\oFF_{\Lambda}R)\to \Mod(\FF_{\Lambda}R)$ is the restriction functor induced by the embedding $\FF_\Lambda R\subseteq \overline{\FF}_\Lambda R$.

	\item The functor $j_!\colon \Mod(\FF_{\Lambda}R)\to \Mod(\oFF_{\Lambda}R)$ is the left adjoint of the restriction functor $j^*$. Explicitly, $j_!$ is the right exact functor $\oFF_{\Lambda}R\otimes_{\FF_{\Lambda}R}-$ which maps the representable functor $\Hom_{\FF_\Lambda R}(\lambda,-)$ to the representable functor $\Hom_{\overline{\FF}_\Lambda R}(\lambda,-),$ for each $\lambda\in \Lambda$ (see \S\ref{subsection:RepresentationsOfCategories}). 

	\item The functor $j_*\colon  \Mod(\FF_\Lambda R)\to \Mod(\overline{\FF}_\Lambda R)$ is defined by \[(j_*M)(\lambda)=\begin{cases}
M(\lambda) & \mbox{ if } \lambda\in \Lambda,\\
0 & \mbox{ if } \lambda=\infty.
\end{cases}\]
\end{itemize}

It is straightforward to see that $(j_!,j^*,j_*)$ is an adjoint triple.

\begin{itemize}	
	\item The functor $i_*\colon \Mod(R)\to \Mod(\overline{\FF}_\Lambda R)$ is the kernel of $j^*$. Explicitly, $i_*$ is given by \[(i_*M)(\lambda)=\begin{cases}
	0 & \mbox{ if } \lambda\in \Lambda,\\
	M & \mbox{ if } \lambda=\infty.
\end{cases}\]

	\item The functor $i^!\colon \Mod(\overline{\FF}_\Lambda R)\to \Mod(R)$ is defined by $i^!(M)=M(\infty).$  Equivalently, $i^!(M)=\ker(\mu_M)$ where $\mu\colon Id_{\Mod(\overline{\FF}_\Lambda R)}\Rightarrow j_*j^*$ is the unit of the adjunction $j^*\dashv j_*$.
	\item The functor $i^*\colon \Mod(\overline{\FF}_\Lambda R)\to \Mod(R)$ is defined by $i^*(M)=M(\infty)/j_!j^*(M)$, thus $i^*(M) =	\coker(\varepsilon_M)$ where $\varepsilon\colon j_!j^*\Rightarrow Id_{\Mod(\overline{\FF}_\Lambda R)}$ is the counit of the adjunction $j_!\dashv j^*$.
\end{itemize}

Again, the triple $(i^*,i_*,i^!)$ is an adjoint triple.

\begin{proposition}\label{proposition:AbelianRecollement}
	The diagram 
\[\xymatrix{
{\Mod R} \ar[rr]|-{i_*} && {\Mod \oFF_\Lambda R} \ar[rr]|-{j^*} \ar@/^/[ll]^-{i^!} \ar@/_/[ll]_-{i^*} && {\Mod \FF_\Lambda R} \ar@/^/[ll]^-{j_*} \ar@/_/[ll]_-{j_!}
}\] is a recollement of abelian categories.
\end{proposition}

\begin{proof}
	Note that $j_*$ is fully faithful. The result now follows from \cite[remark~2.3 and example~2.13]{Psaroudakis14}.
\end{proof}

\subsection{Pregliders and prefragments as torsion-free classes}

Recall from definition \ref{definition:PrefragAndPreglid} that we write $\eta\colon \Prefrag FR\hookrightarrow \Mod(\FF_{\Lambda}R)$ and $\iota\colon \Preglid FR \hookrightarrow \Mod(\oFF_{\Lambda}R)$ for the inclusion functors.  

\begin{proposition}\label{proposition:PreglidIsTorsionfree}
	\begin{enumerate}
		\item The category $\Preglid FR$ is a cotilting torsion-free subcategory of $\Mod(\FF_\Lambda R)$.
		\item The category $\Prefrag FR$ is a cotilting torsion-free subcategory of $\Mod(\oFF_\Lambda R)$.
	\end{enumerate}
\end{proposition}

\begin{proof}
We only show the first statement, the second statement is analogous.  Note that $\Preglid FR$ is closed under subobjects, direct products and extensions in $\Mod(\overline{\FF}_{\Lambda}R)$. By \cite[example~V.2.2]{Stenstrom75}, the category $\Mod(\overline{\FF}_{\Lambda}R)$ is a Grothendieck category. It now follows from \cite[theorem~2.3]{Dickson66} that $\Preglid FR$ is a torsion-free class. Note that the standard projectives in $\Mod(\overline{\FF}_{\Lambda}R)$ are in $\Preglid FR$ (see example \ref{example:StandardProjectives}).  As every projective $\overline{\FF}_{\Lambda}R$-module is a direct summand of a (possibly infinite) direct sum of standard projectives, we find that all projective $\overline{\FF}_{\Lambda}R$-modules are in $\Preglid FR$.  Hence, $\Preglid FR$ is a cotilting torsion-free class.
\end{proof}

By proposition \ref{proposition:PreglidIsTorsionfree}, the functor $\eta$ has a left adjoint $\theta$ and the functor $\iota$ has a left adjoint $\kappa$.  We give an explicit description of the functor $\kappa\colon \Mod(\oFF_{\Lambda}R)\to \Preglid FR.$

\begin{proposition}\label{proposition:Kappa}
The left adjoint functor $\kappa\colon \Mod(\oFF_{\Lambda}R)\to \Preglid FR$ is given by
\[\kappa(M)\colon \lambda \mapsto  \im\left(M(1_{\lambda,\infty})\right).\]
\end{proposition}

\begin{proof}
Note that $\kappa (M) \in \Preglid FR$, for all $M \in \Mod \oFF_\Lambda R.$  There is a natural map $M \to \iota(\kappa(M)).$  It is now easy to verify that a morphism $M \to \iota(N)$ (for any $N \in \Preglid FR$) factors uniquely as $M \to \iota(\kappa(M)) \to \iota(N).$  This shows that $\kappa$ is left adjoint to $\iota.$
\end{proof}

\begin{proposition}\label{proposition:ExactRestrictionOfJ}
The functor $j^*\colon \Mod \oFF_\Lambda R \to \Mod \FF_\Lambda R$ restricts to $j^*\colon \Preglid FR \to \Prefrag FR.$
\end{proposition}

\begin{proof}
Straightforward.
\end{proof}

\begin{corollary}\label{corollary:EtaJAdjoint}
The composition $\kappa \circ j_!\colon \Mod \FF_\Lambda R \to \Preglid FR$ is left adjoint to the composition $\eta \circ j^*\colon \Preglid FR \to \Mod \FF_\Lambda R.$
\end{corollary}

\begin{proof}
This follows directly from $\eta \circ j^* \cong j^* \circ \iota$.
\end{proof}

\begin{proposition}\label{proposition:jL}
The functor $j^*\colon \Preglid FR \to \Prefrag FR$ has a left adjoint $j_L = \kappa \circ j_! \circ \eta.$
\end{proposition}

\begin{proof}
This follows from proposition \ref{proposition:CycleOfAdjoints}\eqref{enumerate:CycleOfAdjointsFFInner} with $H = \kappa \circ j_!$, $G = \eta$, and $F = j^*\colon \Preglid FR \to \Prefrag FR$.
\end{proof}

\subsection{\texorpdfstring{$\Glid FR$ as subcategory of $\Prefrag FR$}{Glid FR as subcategory of Prefrag FR}}

The universal property of the localization functor shows that $j^*\colon \Preglid FR \to \Prefrag FR$ factors as $\Preglid FR \stackrel{Q}{\rightarrow} \Glid FR \stackrel{\phi}{\rightarrow} \Prefrag FR.$  We now show that $\phi$ is fully faithful, so that $\Glid FR$ can be interpreted as a full subcategory of $\Prefrag FR.$

\begin{remark}\label{remark:PhiExplicit}
The functor $\phi\colon \Glid FR \to \Prefrag FR$ is given on objects by mapping a glider $M \in \Ob(\Glid FR) = \Ob(\Preglid FR)$ to the restriction $j^*(M) = M \circ j$.

Let $f\colon M \to N$ be a morphism in $\Glid FR$, say that $f$ is represented by the roof $M \stackrel{s}{\leftarrow} M' \stackrel{g}{\rightarrow} N$ in $\Preglid FR$ (with $s \in \Sigma$).  The corresponding morphism in $\phi(f)\colon \phi(M) \to \phi(N)$ is given by $\phi(f)(\lambda) = g(\lambda) \circ s^{-1}(\lambda)\colon M(\lambda) \to N(\lambda)$, for all $\lambda \in \Lambda.$
\end{remark}

\begin{proposition}\label{proposition:PhiFullyFaithful}\label{proposition:PhiHasLeftAdjoint}
The functor $\phi\colon \Glid FR \to \Prefrag FR$ is fully faithful and admits a left adjoint $\psi\colon \Prefrag FR \to \Glid FR$ given by $\psi = Q \circ j_L.$
\end{proposition}

\begin{proof}
To see that $\phi$ is fully faithful, it suffices to prove that $\eta \circ \phi$ is fully faithful.  Recall that $\Mod \FF_\Lambda R \simeq \Theta^{-1} \Mod \oFF_\Lambda R$ where $\Theta = \{f \in \Mor (\Mod \oFF_\Lambda R) \mid \mbox{$j^*(f)$ is invertible}\}$ (as in proposition \ref{proposition:SigmaRMS}).  Using this equivalence to identify these categories, we can describe the map $\eta \circ \phi\colon \Glid FR \to \Theta^{-1} \Mod \oFF_\Lambda R$ as mapping a roof $M \stackrel{s}{\leftarrow} M' \to N$ (with $s \in \Sigma$) in $\Sigma^{-1} \Preglid FR$ to the same roof in $\Theta^{-1} \Mod \oFF_\Lambda R$ (note that $\Sigma = \Theta \cap \Mor (\Preglid FR)$).  Using lemma \ref{lemma:PhiFullyFaithful}, it is straightforward to show that $\eta \circ \phi$ is fully faithful.

It is shown in proposition \ref{proposition:jL} that $j_L$ is left adjoint to $j^* = \phi \circ Q$.  It now follows from proposition \ref{proposition:CycleOfAdjoints}\eqref{enumerate:CycleOfAdjointsLocOuter} that $\psi = Q \circ j_L$ is left adjoint to $\phi$.
\end{proof}

\subsection{\texorpdfstring{$\Glid FR$ as subcategory of $\Preglid FR$}{Glid FR as subcategory of Preglid FR}}\label{subsection:GlidInPreglid}
In this subsection, we show that the localization $Q\colon \Preglid FR \to \Glid FR$ has a fully faithful left adjoint $L.$  In this way, $\Glid FR$ can be interpreted as a coreflective subcategory of $\Preglid FR.$

\begin{proposition}\label{proposition:LeftAdjoint}
The localization functor $Q\colon \Preglid FR\to \Glid FR$ admits a fully faithful left adjoint $L\colon \Glid FR\to \Preglid FR$.
\end{proposition}

\begin{proof}
Recall from proposition \ref{proposition:jL} that $j_L$ is left adjoint to $j^* = \phi \circ Q$.  It follows from proposition \ref{proposition:CycleOfAdjoints}\eqref{enumerate:CycleOfAdjointsFFInner} (with $F = Q, G = \phi$, and $H = j_L$) that $L = j_L \circ \phi$ is left adjoint to $Q$.   This uses that $\phi$ is fully faithful, see proposition \ref{proposition:PhiFullyFaithful}.  The dual of proposition \ref{proposition:GabrielZisman} implies that $L$ is fully faithful.
\end{proof}

\begin{corollary}\label{corollary:GlidIsCoreflectiveInPreGlid}
The category $\Glid FR$ is a coreflective subcategory of $\Preglid FR.$
\end{corollary}

\begin{proof}
The embedding $L\colon \Glid FR \to \Preglid FR$ has a right adjoint.
\end{proof}

\subsection{A criterion for gliders} We will now give a criterion to determine whether a given prefragment $M \in \Prefrag FR$ is a glider representation, i.e.~whether $M$ lies in the essential image of $\phi\colon \Glid FR \to \Prefrag FR.$

\begin{proposition}\label{proposition:WhenGlider}
The following are equivalent for an $M \in \Prefrag FR:$
\begin{enumerate}
	\item\label{enumerate:WhenGlider1} $M$ is a glider representation,
	\item\label{enumerate:WhenGlider2} the counit of the adjunction $\phi \circ \psi(M) \to M$ is an isomorphism,
	\item\label{enumerate:WhenGlider3} the counit of the adjunction $j^* \circ j_L(M) \to M$ is an isomorphism,
	\item\label{enumerate:WhenGlider4} $j_! \circ \eta (M)$ is a preglider (i.e.~$j_! \circ \eta (M)$ lies in the essential image of $\iota\colon \Preglid FR \to \Mod \oFF_\Lambda R$).
\end{enumerate}
\end{proposition}

\begin{proof}
Note that $M$ is a glider if and only if it lies in the essential image of $\phi$, or equivalently, in the essential image of $j^*\colon \Preglid FR \to \Prefrag FR.$  The equivalence of the first three statements is now easy.

Assume now that \eqref{enumerate:WhenGlider3} holds, so $M \cong j^*(\bar{M})$ where $\bar{M} = j_L(M) \in \Preglid FR.$  We find $j_! \eta (M) \cong j_! \eta j^*(\bar{M}) \cong j_! j^* \iota (\bar{M}).$  As $\iota(M)$ is a preglider, the counit $j_! j^* \iota (\bar{M}) \to \iota(M)$ satisfies the conditions of lemma \ref{lemma:PhiFullyFaithful}.  Hence, $j_! \eta (M)$ is a preglider.

Assume now that \eqref{enumerate:WhenGlider4} holds, so $j_! \eta (M) \cong \iota(\bar{M})$, for some $\bar{M} \in \Preglid FR.$    Using that $j^* \circ j_! \cong 1$, we find: $\eta(M) \cong j^* \circ j_! \circ \eta (M) \cong j^* \circ \iota(\bar{M}) \cong \eta \circ j^* (\bar{M}).$  As $\eta$ is fully faithful, we find $M \cong j^* (\bar{M}),$ hence $M$ is a glider representation.
\end{proof}

We now revisit the example in the introduction.

\begin{example}\label{Example:IntroExample}
Let $R = \bR[t]$ and let $FR$ be the usual $\bZ$-filtered ring with $t$ is degree 1.  Let $\Lambda = \bZ^{\leq 0}$ and let $M \in \Mod \FF_\Lambda R$ be the functor given by $-i \mapsto C^i(\bR)$ and $t^n\colon C^i(\bR) \to C^{i-n}(\bR)\colon f \mapsto \frac{d^n}{dx^n} f.$  This is clearly a prefragment.  To see that this is a glider representation, we consider the $\bR[t]$-module
\[\Omega = \oplus_{i \geq 0} C^0(\bR) \cdot t^i.\]
Let $N$ be the submodule of $\Omega$ generated by the elements $f\cdot t^i - (\frac{d}{dx} f) t^{i-1}$ where $f \in C^1(\bR), i \geq 1.$  We can extend $M$ to a preglider $\bar{M}$ by setting $\bar{M}(\infty) = \Omega/N$.  It is now clear that the restriction $M = j^*(\bar{M})$ is a glider representation.
\end{example}

%% file: MainProperties.tex
\section{Properties of the category of glider representations}\label{section:PropertiesOfTheCategoryOfGliderRepresentations}

Let $\Gamma$ be an ordered group and let $\Lambda \subseteq \Gamma$ be any subset.  Let $FR$ be a $\Gamma$-filtered ring.  In this section, we establish some (homological) properties of the category of glider representations.  We start by showing that $\Preglid FR$ and $\Prefrag FR$ are complete and cocomplete.  The category $\Glid FR$ inherits these properties as it is a reflective subcategory of $\Prefrag FR$ (via the functor $\phi$) and a coreflective subcategory of $\Preglid FR$ (via the functor $L$).

In \S\ref{subsection:GlidAsQuotient}, we endow the category $\Glid FR$ with the structure of a deflation-exact category, induced by the natural exact structures on $\Preglid FR$ and $\Prefrag FR.$  We subsequently show that the conflation structure on $\Glid FR$ is an inherent feature (i.e.~it can be defined without referring to $\Preglid FR$ or $\Prefrag FR$) by showing that $\Glid FR$ is a (Grothendieck) deflation quasi-abelian category. 

\subsection{\texorpdfstring{Limits and colimits in $\Glid FR$}{Limits and colimits in Glid FR}} We will describe limits and colimits in $\Glid FR$ using the embeddings $L$ and $\phi$ into the categories $\Preglid FR$ and $\Prefrag FR,$ respectively (see \S\ref{section:StructureResults}).

\begin{proposition}\label{proposition:PreglidIsQuasiAbelian}
The categories $\Preglid FR$ and $\Prefrag FR$ are complete and cocomplete quasi-abelian categories.
\end{proposition}

\begin{proof}
It follows from theorem \ref{theorem:BondalVandenBergh} and proposition \ref{proposition:PreglidIsTorsionfree} that $\Preglid FR$ and $\Prefrag FR$ are quasi-abelian.  It follows directly from the definitions that $\Preglid FR$ and $\Prefrag FR$ are closed under arbitrary direct sums and products.  We conclude that $\Preglid FR$ and $\Prefrag FR$ are complete and cocomplete.
\end{proof}

\begin{remark}
Limits in $\Preglid FR$ and $\Prefrag FR$ coincide with limits in $\Mod \oFF_\Lambda R$ and $\Mod \FF_\Lambda R$, respectively; colimits, however, need not coincide.
\end{remark}

\begin{proposition}\label{proposition:LimitsAndColimits}
The category $\Glid FR$ is complete and cocomplete.  Specifically, let $D\colon J \to \Glid FR$ be a diagram.
\begin{enumerate}
  \item We have $\lim D \cong \psi(\lim \phi D)$ and $\colim D \cong \psi(\colim \phi D).$
	\item We have $\lim D \cong Q(\lim (LD))$ and $\colim D \cong Q(\colim (LD)).$
\end{enumerate}
\end{proposition}

\begin{proof}
We have shown in proposition \ref{proposition:PreglidIsQuasiAbelian} that $\Preglid FR$ and $\Prefrag FR$ are complete and cocomplete.  As $\Glid FR$ is a reflective subcategory of $\Prefrag FR$ (proposition \ref{proposition:PhiHasLeftAdjoint}) and a coreflective subcategory of $\Preglid FR$ (corollary \ref{corollary:GlidIsCoreflectiveInPreGlid}), the other statements follow.
\end{proof}

\begin{remark}
As $\Glid FR$ is a reflective subcategory of $\Prefrag FR$ via $\phi$, limits in $\Glid FR$ and $\Prefrag FR$ coincide.  Similarly, $\Glid FR$ is a coreflective subcategory of $\Preglid FR$ via $L$, and, as such, colimits in $\Glid FR$ and $\Preglid FR$ coincide.
\end{remark}

\subsection{\texorpdfstring{$\Glid FR$ is a deflation-exact category}{Glid FR is a deflation-exact category}}\label{subsection:GlidAsQuotient} The category $\Glid FR$ naturally factors the map $j^*\colon \Preglid FR \stackrel{Q}{\rightarrow} \Glid FR \stackrel{\phi}{\rightarrow} \Prefrag FR$.  As such, it can be endowed with a conflation structure, either as a localization of $\Preglid FR$ or as a subcategory of $\Prefrag FR.$  In this subsection, we start with the former (corollary \ref{corollary:GlidersAreDeflationExact}) and show that it coincides with the latter (proposition \ref{proposition:GlidAsSubcatOfPrefrag}).

In definition \ref{definition:GlidAsLoc}, we introduced the category of gliders via a localization functor $Q\colon \Preglid FR \to \Sigma^{-1} \Preglid FR = \Glid FR$.  We will now interpret $\Glid FR$ as the quotient $\Preglid FR / i_*(\Mod R)$, as defined in \S\ref{subsection:LocalizationsPreliminaries}.

\begin{proposition}\label{proposition:GlidAsLocalization}
Let $(\Gamma, \leq)$ be an ordered group and let $FR$ be a $\Gamma$-filtered ring.  Let $\Lambda \subseteq \Gamma$ be any subset. 
\begin{enumerate}
	\item With the embedding $i_*$, the category $\Mod(R)$ is an admissibly deflation-percolating subcategory of $\Preglid FR$.
	\item A morphism $f\colon M\to N$ in $\Preglid FR$ is a weak $\Mod(R)$-isomorphism if and only if $f \in \Sigma$ (i.e.~the induced maps $f_{\lambda}\colon M(\lambda)\to N(\lambda)$ are isomorphisms for each $\lambda\neq \infty$).
	\item $\Glid FR = \Sigma^{-1} \Preglid FR = \Preglid FR / i_*(\Mod R).$
\end{enumerate}
\end{proposition}

\begin{proof}
As $\Preglid FR$ is an extension-closed subcategory of the abelian category $\Mod \oFF_\Lambda R$ (see proposition \ref{proposition:PreglidIsTorsionfree}), we find that $\Preglid FR$ is an exact category (see \cite[lemma 10.20]{Buhler10}): the conflations of $\Preglid FR$ are given by the short exact sequences in $\Mod \oFF_\Lambda R$.
	
We now verify axioms \ref{A1}, \ref{A2} and \ref{A3} of definition \ref{definition:GeneralPercolatingSubcategory}. As the category $\Preglid FR$ is an exact category, axiom \ref{A3} is automatic (see remark \ref{remark:PercolatingRemark}). Axiom \ref{A1} is straightforward to see. For axiom \ref{A2}, consider a map $f\colon M \to A$ where $M\in \Preglid FR$ and $A\in i^*(\Mod R)\subseteq \Preglid FR$. Clearly $f_{\lambda}\colon M(\lambda)\to 0$ for each $\lambda\in \Lambda$. The map $f_{\infty}\colon M(\infty)\to A$ has an epi-mono factorization as $\Mod(R)$ is abelian.  It readily follows that $f$ is admissible in $\Preglid FR$.  Hence, $\Mod(R)$ is an admissibly deflation-percolating subcategory of $\Preglid FR$.
	
	For the next statement, let $f\colon M\to N$ be a morphism in $\Preglid FR$.  Assume that $f$ is a weak $\Mod(R)$-isomorphism.  As $f$ is admissible, we find that $\ker(f),\coker(f)\in \Mod(R)$.  It follows that $\ker(f_{\lambda}) = 0 = \coker(f_{\lambda})$ for each $\lambda\neq \infty$.  Hence, each $f_{\lambda}$ is an isomorphism.  Conversely, assume that $f_{\lambda}$ is an isomorphism for each $\lambda\neq \infty$. Clearly, $\ker(f)$ and $\coker(f)$ exist, and belong to $\Mod(R)$.  The fact that $f$ is admissible follows immediately from the fact that $f_{\infty}$ admits an epi-mono factorization.
	
	For the last statement, the first equality is the definition of $\Glid FR$, the second equality is the definition of $\Preglid FR / i_*(\Mod R).$
\end{proof}

It follows from proposition \ref{proposition:GlidAsLocalization} that the localization $\Sigma^{-1} \Preglid FR$ can be described using theorem \ref{theorem:MainTheoremPercoLocalization}.

\begin{corollary}\label{corollary:GlidersAreDeflationExact}
The weakest conflation structure on $\Glid FR \coloneqq \Sigma^{-1} \Preglid FR$ for which the localization functor $Q\colon \Preglid FR \to \Glid FR$ is conflation-exact, gives $\Glid FR$ the structure of a deflation-exact category.
\end{corollary}

\begin{remark}
Theorem \ref{theorem:MainTheoremPercoLocalization} also implies that $\Sigma$ is a saturated right multiplicative system, i.e. we recover proposition \ref{proposition:SigmaRMS}.
\end{remark}

The following proposition connects the conflation structures of $\Glid FR$ and $\Prefrag FR$ using the embedding $\phi.$

\begin{proposition}\label{proposition:GlidAsSubcatOfPrefrag}
The embedding $\phi\colon \Glid FR \to \Prefrag FR$ reflects conflations, i.e.~a sequence $K\xrightarrow{f}M\xrightarrow{g}N$ in $\Glid FR$ is a conflation if and only if $\phi(K)\xrightarrow{\phi(f)}\phi(M)\xrightarrow{\phi(g)}\phi(N)$ is a conflation in $\Prefrag FR$.
\end{proposition}

\begin{proof}
By construction, $\phi$ is conflation-exact and hence maps conflations to conflations.  Thus, let $K\xrightarrow{f}M\xrightarrow{g}N$ be a sequence in $\Glid FR$ which maps to a conflation in $\Prefrag FR.$  As $\phi$ is fully faithful, we know that $K \to M \to N$ is a kernel-cokernel pair in $\Glid FR.$

Using that $L$ is a left adjoint, we find that $L(g)\colon L(M) \to L(N)$ is the cokernel of $L(f).$  Since $\Preglid FR$ is a quasi-abelian category, $L(g)$ is a deflation.  Hence, so is $g = Q \circ L (g).$  This establishes that $K \inflation M \deflation N$ is a conflation in $\Glid FR.$
\end{proof}

\begin{proposition}\label{proposition:phiIsEpiReflective}
For all $M \in \Prefrag FR$, the reflection $M \to \phi \circ \psi(M)$ is a deflation.
\end{proposition}

\begin{proof}
We write $\Omega_M$ for the $R$-module $j_! \circ \eta(M)(\infty).$  As $\psi = j_L \circ Q = \eta \circ j_! \circ \kappa \circ Q$, we find $\phi \circ \psi(M)(\lambda) \cong \im (M(1_{\lambda, \infty})\colon M(\lambda) \to \Omega_M)$, and the reflection is given by maps $r_M\colon M \to \phi \circ \psi(M)$ for which the following diagram commutes
\[\xymatrix{
 M(\lambda) \ar[rr]^{M(1_{\lambda, \infty})} \ar[d]^{r_M(\lambda)} && \Omega_M \ar@{=}[d] \\
 \im (M(1_{\lambda, \infty})) \ar[rr] && \Omega_M} \]
It follows that each $r_M$ is an epimorphism.  As such the map $r_M$ is an epimorphism in $\Mod \oFF_\Lambda R$ and hence $\ker r_M \inflation M \deflation \phi \circ \psi(M)$ is a conflation.
\end{proof}

\begin{remark}\label{remark:GlidersEpireflective}
Via the functor $\phi$, the category $\Glid FR$ is an \emph{epi-reflective} subcategory of $\Prefrag FR,$ i.e.~the fully faithful embedding has a left adjoint, and the unit $1 \to \phi \circ \psi$ of the adjunction is a deflation.  Hence, every prefragment $M \in \Prefrag FR$ has a largest glider quotient object $\phi \circ \psi (M).$
\end{remark}

\begin{corollary}\label{corollary:GlidersClosedUnderSubobjects}
The essential image of $\phi$ is closed under subobjects.
\end{corollary}

\begin{proof}
Let $M \in \Glid FR$ and let $f\colon N \hookrightarrow \phi M$ be a monomorphism in $\Prefrag FR.$  We find the following commutative diagram
\[\xymatrix{
N \ar[r]^{f} \ar[d]^{r_N} & {\phi (M)} \ar[d]^{\cong} \\
{\phi \psi(N)} \ar[r] & {\phi\psi\phi(M)}  }\]
where $r\colon 1 \to \phi \circ \psi$ is the unit of the adjunction.  As the top-right branch composes to a monomorphism, so does the left-lower branch.  In particular, $r_N\colon N \to \phi\psi N$ is a monomorphism.  As $r_N$ is a deflation (see proposition \ref{proposition:phiIsEpiReflective}), we find that $r_N$ is an isomorphism.  Hence, $N$ lies in the essential image of $\phi.$
\end{proof}

\subsection{Gliders as a Grothendieck deflation quasi-abelian category}\label{subsection:GrothendieckQuasiAbelian}

We now proceed to showing that $\Glid FR$ is a Grothendieck deflation quasi-abelian category (see definition \ref{definition:Grothendieck} below).  We start by showing that $\Glid FR$ is deflation quasi-abelian (see definition \ref{definition:DeflationQuasiAbelian}).

\begin{theorem}\label{theorem:GlidIsDeflationQuasiAbelian}
The category $\Glid FR$ is a complete and cocomplete deflation quasi-abelian category.
\end{theorem}

\begin{proof}
We have already established that $\Glid FR$ is deflation-exact (see corollary \ref{corollary:GlidersAreDeflationExact}), and is complete and cocomplete (see proposition \ref{proposition:LimitsAndColimits}).  In particular, $\Glid FR$ is pre-abelian.  To see that $\Glid FR$ is deflation quasi-abelian, we only need to verify that the conflation structure of $\Glid FR$ consists of all kernel-cokernel pairs.  For this, it suffices to show that all cokernels are deflations.  Let $f\colon X \to Y$ be any map; we claim that $g\colon Y \to \coker f$ is a deflation.  Applying the functor $L$ and using that a left adjoint commutes with cokernels, we find that $L(g)$ is the cokernel of $L(f)$.  As $\Preglid FR$ is a quasi-abelian category, this means that $L(g)$ is a deflation.  Finally, as $Q$ maps deflations to deflations, we find that $QL(g) \cong g$ is a deflation in $\Glid FR.$
\end{proof}

\begin{remark}\label{remark:ConflationStructureOnGlidNatural}
It follows from theorem \ref{theorem:GlidIsDeflationQuasiAbelian} that the conflation structure of $\Glid FR$ is not additional structure but is inherent to the category: the class of conflations consists of all kernel-cokernel pairs.
\end{remark}

We recall the definition of a Grothendieck quasi-abelian category from \cite{Wallbridge15}.

Let $\CC$ be a category admitting all small filtered direct limits.  We say that an object $C$ is \emph{finitely presentable} if, for all filtered diagrams $D\colon J \to \CC$, the natural morphism $\varinjlim \Hom(C,D) \to \Hom(C, \varinjlim D)$ is an isomorphism.  We say the category $\CC$ is \emph{locally presentable} if it has all small filtered direct limits and every object in $\CC$ is a filtered direct limit of finitely presentable objects.

\begin{definition}\label{definition:Grothendieck}
A (deflation) quasi-abelian category is called a \emph{Grothendieck (deflation) quasi-abelian category} if it is locally presentable and has exact filtered direct limits.
\end{definition}

\begin{proposition}\label{proposition:PreglidPrefragAreGrothendieck}
The categories $\Preglid FR$ and $\Prefrag FR$ are Grothendieck quasi-abelian categories.
\end{proposition}

\begin{proof}
It is shown in proposition \ref{proposition:PreglidIsQuasiAbelian} that $\Preglid FR$ and $\Prefrag FR$ are cocomplete.  As direct limits in $\Mod \oFF_\Lambda R$ and $\Mod \oFF_\Lambda R$ are exact and taken pointwise, we find that filtered direct limits are exact $\Preglid FR$ and $\Prefrag FR$ as well.  As the standard projectives in $\Preglid FR$ and $\Prefrag FR$ form a strong generating set, it follows from \cite[theorem 1.11]{AdamekRosicky94} that $\Preglid FR$ and $\Prefrag FR$ are finitely presentable.
\end{proof}

\begin{theorem}\label{theorem:Grothendieck}
Let $\Gamma$ be an ordered group.  Let $\Lambda \subseteq \Gamma$ be any subset and let $FR$ be a $\Gamma$-filtered ring.  The category $\Glid FR$ of glider representations is a Grothendieck deflation quasi-abelian category.
\end{theorem}

\begin{proof}
Similar to the proof of proposition \ref{proposition:PreglidPrefragAreGrothendieck}.
\end{proof}

\subsection{Some examples and counterexamples}  We provide some examples and counterexamples to illustrate some properties in this section.

By proposition \ref{proposition:PreglidIsQuasiAbelian}, the category $\Preglid FR$ admits kernels and cokernels.  For sake of reference we describe these explicitly in the next proposition (see also \cite[corollary~3.6]{SchapiraSchneiders16} for a similar result).

\begin{proposition}\label{proposition:ExcplicitKernelCokernelInPreglid}
	Let $f\colon M\to N$ be a morphism in $\Preglid FR$.  For each $\lambda\in \Ob(\oFF_{\Lambda}R)$ we have the following canonical isomorphisms:
	\begin{enumerate}[(a)]
		\item $(\ker f)(\lambda)\cong \ker f(\lambda)$,
		\item $(\coker f)(\lambda) \cong \im(N(\lambda)\to \coker f_{\infty})$,
		\item $(\im f)(\lambda)\cong \ker (N(\lambda)\to \coker f_{\infty})$,
		\item $(\coim f)(\lambda) \cong \im f(\lambda)$.
	\end{enumerate}	
\end{proposition}

\begin{proof}
	This follows from the fact that $\iota$ commutes with kernels and that $\kappa$ commutes with cokernels.
\end{proof}

\subsubsection{The functor $\phi$ need not commute with colimits}

In proposition \ref{proposition:GlidAsSubcatOfPrefrag}, we showed that the functor $\phi$ reflects conflations.  However, $\phi$ need not commute with colimits, as the following example shows.  In particular, one needs to be careful in computing colimits using the functor $\phi.$

\begin{example}\label{example:SymmetricGroup1}
Let $R=\mathbb{C}[S_3]$ and consider the $\mathbb{Z}^{\geq 0}$-filtration 
\[\mathbb{C}\subseteq \mathbb{C}[S_3]\subseteq \mathbb{C}[S_3]\subseteq \dots\]
	and let $\Lambda=\left\{-1,0\right\}$. Let $V$ be a two-dimensional complex vector space with basis $e_a,e_b$ and set $e_c=-(e_a+e_b)$. View $S_3$ as the permutations on the set $\left\{a,b,c\right\}$ and let $S_3$ act on the vectors $e_a,e_b,e_c$ accordingly. Thus, $V$ is the standard representation of $S_3$.  Consider the following conflation of prefragments:	
	\[\xymatrix{
		M\ar@{>->}[d]_f&&0\ar@{^{(}->}[r]\ar[d] & \mathbb{C}e_a\ar@{^{(}->}[d]\\
		N\ar@{->>}[d]_g&&\mathbb{C}e_b\ar@{^{(}->}[r]\ar@{=}[d] & V\ar[d]\\
		S&& \mathbb{C}e_b\ar@{^{(}->}[r] & V/\mathbb{C}e_a
	}\]
	The reader may verify that $M$ and $N$ belong to $\phi(\Glid FR)$.  However, $S$ cannot be extended to a glider representation. Indeed,	consider the transposition $(ab)\in \Hom_{\FF_{\Lambda}R}(-1,0)=\mathbb{C}[S_3]$. As $g\colon M\to N$ is a natural transformation, the following diagram commutes:
	\[\xymatrix{
		\mathbb{C}e_b\ar[r]^{N(ab)}\ar@{=}[d] & V\ar[d]\\
		\mathbb{C}e_b\ar[r]_{S(ab)} & V/\mathbb{C}e_a
	}\] Hence, $S(ab)$ acts on $e_b$ as zero. It follows that $S$ is not a glider representation as the partial actions are not induced by a $\mathbb{C}[S_3]$-module.  Indeed, the transposition $(ab)$ squares to the identity and thus cannot act as zero.
	
	Following proposition \ref{proposition:LimitsAndColimits}, the cokernel of $f$ as a morphism of glider representations is zero. In other words, the inflation $f$ of prefragments is not an inflation of glider representations (as it is not the kernel of its cokernel).
\end{example}

\subsubsection{The functor $L$ need not be conflation-exact} The embedding $L\colon \Glid FR \to \Preglid FR$ is right exact (as it is left adjoint to $Q$), but need not commute with limits.  The following example shows that $L$ need not even be conflation-exact.

\begin{example}\label{example:LeftAdjointNotExact}
	Let $\Gamma=\mathbb{Z}$, $\Lambda=\left\{0\right\}$ and let $FR$ be the $\mathbb{Z}^{\geq 0}$-filtration of $\mathbb{C}[t,x]/(xt)$ given by
	\[\mathbb{C}[t]\subseteq \mathbb{C}[t,x]/(xt)\subseteq \mathbb{C}[t,x]/(xt) \subseteq \dots\]
So, $F_0 R = \mathbb{C}[t]$, and $F_i R = \mathbb{C}[t,x]/(xt)$, for all $i \geq 1.$

  Consider the natural morphism $M\xrightarrow{g} N$ in $\Preglid FR$ determined by the solid part of the commutative diagram
	\[\xymatrix{
		\ker(g)\ar@{>.>}[d]_f 	&& t\mathbb{C}[t] \ar@{^{(}.>}[r] \ar@{^{(}.>}[d]^{t}			& t\mathbb{C}[t,x]/(xt) \ar@{^{(}.>}[d]\\
		M\ar@{->>}[d]_g 			&& \mathbb{C}[t] 	\ar@{^{(}->}[r]	\ar@{->>}[d]				& \mathbb{C}[t,x]/(xt) \ar@{->>}[d]\\
		N											&& \mathbb{C}			\ar@{^{(}->}[r]											& \mathbb{C}[t,x]/(t)		
	}\]
	Using proposition \ref{proposition:ExcplicitKernelCokernelInPreglid}, one readily verifies that $\ker(g)\xrightarrow{f} M\xrightarrow{g} N$ is a conflation in $\Preglid FR$ and thus descends to a conflation in $\Glid FR$ as well.
	
	Applying $L\circ Q$ to $f$ we obtain the following commutative diagram in $\Preglid FR$:
	\[\xymatrix{
	LQ(\ker(g))\ar[d]_{LQ(f)} && \mathbb{C}[t]\ar@{^{(}->}[r]\ar@{^{(}->}[d]^t		 & \mathbb{C}[t,x]/(xt)\ar[d]^{t}	\\
	LQ(M) && \mathbb{C}[t]\ar@{^{(}->}[r]	 & \mathbb{C}[t,x]/(xt)
	}\]
	Note that multiplication by $t$ is not an injection from $\mathbb{C}[t,x]/(xt)$ to itself.  Hence, $LQ(\ker g) \to LQ(M)$ is not a monomorphism.  This shows that $L$ does not commute with kernels of deflations.  In particular, $L$ is not a conflation-exact functor
\end{example}

\begin{remark}
Although the functor $L$ is not conflation-exact, it maps deflations to deflations (as it maps cokernels to cokernels and $\Preglid FR$ is quasi-abelian).  Thus, given a conflation $X \inflation Y \deflation Z$ of glider representations, we obtain a sequence $LX \to LY \deflation LZ$ of pregliders.  Let $K$ be the kernel of $LY \deflation LZ.$  As the natural map $LX \to K$ becomes an isomorphism under the functor $Q$ (this uses that $Q$ is conflation-exact and that $Q \circ L \cong 1$), the morphism $LX \to K$ is a weak isomorphism.  So, for every $\lambda \in \Lambda,$ the map $LX(\lambda) \to K(\lambda)$ is an isomorphism.
\end{remark}

\subsubsection{The category $\Glid FR$ need not be (2-sided) exact}  We now give two examples to show that $\Glid FR$ is not a (2-sided) exact structure.  In particular, $\Glid FR$ need not be a quasi-abelian (or abelian) category.

\begin{example}\label{example:GlidIsNotExact}
	Let $R=k[t]$ be a polynomial ring in one variable $t$ over some field $k$. Let $\Gamma=\mathbb{Z}$ and let $\Lambda=\left\{-1,0\right\}$. Let $FR$ be the $\mathbb{Z}^+$-filtration 
		\[k\subseteq k[t]^{\leq 1}\subseteq k[t]^{\leq 2} \subseteq k[t]^{\leq 3} \subseteq \dots\]
		The following commutative diagram defines morphisms $f,g$ and $h$ in $\Preglid FR$.
		\[\xymatrix{
	L \ar@{>->}[d]_{f} && 0\ar@{^{(}->}[r]\ar[d] & kt\ar@{^{(}->}[r]\ar[d]^{\begin{psmallmatrix}1\\0\end{psmallmatrix}} & tk[t]\ar[d]^{\begin{psmallmatrix}1\\0\end{psmallmatrix}}\\
	L\oplus L\ar[d]_{g}^{\rotatebox{90}{$\sim$}} && 0 \ar@{^{(}->}[r]\ar[d] & kt \oplus kt \ar@{^{(}->}[r]\ar[d]^{\begin{psmallmatrix}1&0\\0&t\end{psmallmatrix}} & tk[t]\oplus tk[t]\ar[d]^{\begin{psmallmatrix}1&t\end{psmallmatrix}}\\
	M \ar@{>->}[d]_{h}&& 0 \ar@{^{(}->}[r]\ar[d] & kt\oplus kt^2 \ar@{^{(}->}[r]\ar[d]^{\begin{psmallmatrix}0&0\\1&0\\0&1 \end{psmallmatrix}} & tk[t]\ar[d]_1\\
	N && k \ar@{^{(}->}[r] & k\oplus kt \oplus kt^2 \ar@{^{(}->}[r] & k[t]
}\]
 Clearly $f$ and $h$ are inflations in $\Preglid FR$. Note that $\ker(g),\coker(g)\in \Mod(k[t])$ and hence $g$ is a weak $\Mod(k[t])$-isomorphism by proposition \ref{proposition:GlidAsLocalization}. By definition, the maps $Q(g\circ f)$ and $Q(h)$ are inflations in $\Glid FR$. We claim that the composition $Q(h)\circ Q(g\circ f)$ is not an inflation in $\Glid FR$. It follows that $\Glid FR$ does not satisfy axiom \ref{L1} and hence $\Glid FR$ is not an exact category.

Assume that $Q(h)\circ Q(g\circ f)=Q(h\circ g\circ f)$ is an inflation in $\Glid FR$. By proposition \ref{proposition:ExcplicitKernelCokernelInPreglid}, the cokernel of $hgf$ in $\Preglid FR$ is given by $k\hookrightarrow k \hookrightarrow k$ and $\ker(\coker(hgf))$ is given by $0\hookrightarrow kt\oplus kt^2\hookrightarrow tk[t]$. It follows that $hgf$ is not the kernel of its cokernel in $\Preglid FR$ and thus that $hgf$ is not an inflation in $\Glid FR$.
\end{example}

\begin{example}\label{example:SymmetricGroup2}
	In this example, we provide another illustration of the failure of axiom \ref{L1} in $\Glid FR$.  Using the same notation as in example \ref{example:SymmetricGroup1}, consider the following morphisms of prefragments:
	\[\xymatrix{
		A\ar@{>->}[d]^f&&0\ar@{^{(}->}[r]\ar[d] & \mathbb{C}e_a\ar[d]\\
		B\ar@{>->}[d]^g&&0\ar@{^{(}->}[r]\ar[d] & V\ar[d]^{\begin{psmallmatrix}1\\0\end{psmallmatrix}}\\
		C\ar@{->>}[d]^h&&\mathbb{C}e_b\ar@{^{(}->}[r]^{\Delta}\ar@{=}[d] & V\oplus V\ar[d]^{\begin{psmallmatrix}0&1\end{psmallmatrix}}\\
		D&&\mathbb{C}e_b\ar@{^{(}->}[r] & V
	}\] where $\Delta\colon V \to V \oplus V$ is the diagonal map.  It is easy to see that $f$ and $g$ can be extended to inflations in $\Glid FR$ but, as in example \ref{example:SymmetricGroup1}, the composition $g\circ f$ is not an inflation in $\Glid FR$.  Hence, the composition of inflations need not be an inflation in $\Glid FR$.
\end{example}

\subsubsection{The functors $j_*, j_!$ do not restrict to functors $\Prefrag FR \to \Preglid FR$} It is shown in proposition \ref{proposition:AbelianRecollement} that the top row in figure \ref{figure:TheDiagram} is a recollement.  The functors $j_*, j_!$, in general, fail to restrict to functors $\Prefrag FR \to \Preglid FR$.  The following example is based on \cite[example 1.3.3]{CaenepeelVanOystaeyenBook19}.

\begin{example}\label{example:NoRestriction}
	Set $\Gamma=\mathbb{Z}, \Lambda=-\mathbb{N}, R=\mathbb{Q},$ and set \[FR=\mathbb{Z}\subseteq \mathbb{Q}\subseteq \mathbb{Q}\subseteq \dots\]
	Let $M\in \Prefrag FR$ be the cokernel of the obvious map $\Hom_{\FF_\Lambda R}(0,-)\xrightarrow{\times 2} \Hom_{\FF_\Lambda R}(0,-)$, explicitly, $M$ is determined by 
	\[M(\lambda)=\begin{cases}0 & \mbox{ if } \lambda\neq 0,\\
	\mathbb{Z}/2\mathbb{Z} &\mbox{ if } \lambda=0.\end{cases}\]
Note that $j_!(M)(\infty)=0$ and thus $j_!(M)\notin \Preglid FR$. This shows that $j_!$ does not restrict to a functor $\Prefrag FR \to \Preglid FR$.

Note that $j_*(M)(\infty)=0$ but $j_*(M)(0)=\mathbb{Z}_2$, it follows that $j_*(M)\notin \Preglid FR$. Hence, the functor $j_*$ does not restrict to a functor $\Prefrag FR\to \Preglid FR$. 
\end{example}

\subsubsection{$\Prefrag FR$ is not the exact hull of $\Glid FR$}  It is shown in \cite{Rosenberg11} (see also \cite{HenrardVanRoosmalen19b}) that a deflation-exact category $\CC$ can be embedded in an exact category $\overline{\CC}$ in a 2-universal way: there is a conflation-exact embedding $i\colon \CC \to \overline{\CC}$ such that, for each exact category $\EE,$ the natural functor 
\[- \circ i\colon \Hom_{\mathrm{exact}}(\overline{\CC}, \EE) \to \Hom_{\mathrm{exact}}(\CC, \EE)\]
is an equivalence.  In particular, every conflation-exact functor $\CC \to \EE$ factors essentially uniquely through $i\colon \CC \to \overline{\CC}.$

It might now be tempting to assume that $\Prefrag FR$ is the exact hull of $\Glid FR$ (especially in light of proposition \ref{proposition:DerivedEquivalences} below).  However, in the notation of example \ref{example:NoRestriction}, the only sub-prefragments of $M$ are $0$ and $M$ itself.  As $M$ is not a glider, we see that $M$ cannot occur as an extension of gliders.  Hence, $\Prefrag FR$ is not the exact hull of $\Glid FR$.

\section{Noetherian objects}\label{section:Noetherian}

So far, we have considered the category of all glider representations of a filtered ring $FR$.  In this section, we look at Noetherian objects in deflation quasi-abelian categories.  Our first result is theorem \ref{theorem:NoetherianObjects}, stating that the subcategory of Noetherian objects is a Serre subcategory (see definition \ref{definition:GeneralPercolatingSubcategory}), and hence itself quasi-abelian (see \cite[lemma 4]{Rump01}).  We then provide a number of equivalent formulations of when an object of $\Glid FR$ is Noetherian in proposition \ref{proposition:WhenNoetherian}.

We will use the following proposition, which is a straightforward adaptation of a similar statement in for quasi-abelian categories (\cite[proposition 1.1.4]{Schneiders99}, see also \cite[proposition 1]{Rump01}).

\begin{proposition}\label{proposition:CoimageFactorization}
In a deflation quasi-abelian category, a morphism $f\colon X \to Y$ factors as $X \stackrel{d}\deflation \coim f \stackrel{i}{\hookrightarrow} Y$ where $d$ is a deflation and $i$ is a monomorphism.
\end{proposition}

\subsection{Noetherian objects in deflation quasi-abelian categories}
We now look at the subcategory of Noetherian objects in a deflation quasi-abelian category.  Recall that an object $X$ in a category is called \emph{Noetherian} if any ascending sequence of subobjects of $X$ is stationary.  The following theorem is a straightforward adaptation of a similar result for abelian categories.

\begin{theorem}\label{theorem:NoetherianObjects}
Let $\CC$ be a deflation quasi-abelian category.  The full subcategory $\NN \subseteq \CC$ of Noetherian objects is a Serre subcategory.  In particular, $\NN$ is deflation quasi-abelian.
\end{theorem}

\begin{proof}
Consider a conflation $X \inflation Y \deflation Z$ in $\CC$ with $Y \in \NN.$  As the composition of monomorphisms is a monomorphism, we find that $X \in \NN.$  To show that $Z \in \NN,$ consider an ascending sequence $Z_0 \hookrightarrow Z_1 \hookrightarrow \cdots$ of subobjects of $Z.$  Taking the pullback along $Y \deflation Z$ (and using that pullbacks of monomorphisms are again monomorphisms), we find an ascending sequence $Y_0 \hookrightarrow Y_1 \hookrightarrow \cdots$ of subobjects of $Y:$
\[\begin{tikzcd}
X \arrow[r, tail] \arrow[d, equal] & Y_i \arrow[r, two heads] \arrow[d] & Z_i \arrow[d] \\
X \arrow[r, tail]           & Y \arrow[r, two heads]                        & Z
\end{tikzcd}\]
As $Y \in \NN$, we know that $Y_i \hookrightarrow Y$ is an isomorphism for $i \gg 0$.  It follows that $Z_i \hookrightarrow Z$ is an isomorphism as well.

Assume now that $X,Z \in \NN.$  We show that $Y \in \NN$.  Let $Y_0 \hookrightarrow Y_1 \hookrightarrow \cdots$ be an ascending sequence of subobjects of $Y.$  Taking pullbacks along $X \inflation Y$ yields a diagram
\[\begin{tikzcd}
X \cap_Y Y_i \arrow[r, tail] \arrow[d] & Y_i \arrow[r, two heads] \arrow[d] & Y / (X \cap_Y Y_i) \arrow[d] \\
X \arrow[r, tail]           & Y \arrow[r, two heads]                        & Z
\end{tikzcd}\]
where the downward arrows are monomorphisms.  Indeed, the left-most arrow is a monomorphism as it is the pullback of a monomorphism.  For the right-most arrow, it follows from the pullback property that the map $X \cap_Y Y_i \inflation Y_i$ is the kernel of the composition $Y_i \to Y \deflation Z$, so that $Y_i \deflation Y / (X \cap_Y Y_i)$ is the coimage of the composition $Y_i \to Y \deflation Z$.  Proposition \ref{proposition:CoimageFactorization} now shows that the right-most map in the above diagram is a monomorphism.

As $X,Z \in \NN$, we may assume the outer monomorphisms are isomorphisms for $i \gg 0.$  It now follows from the short five lemma (\cite[lemma 5.3]{BazzoniCrivei13} or \cite[lemma 3]{Rump01}) that $Y_i \hookrightarrow Y$ is an isomorphism as well.  Hence, $Y \in \NN$.
\end{proof}

\subsection{Noetherian glider representations}

For later use, we introduce the category of Noetherian prefragments and glider representations.

\begin{definition}\label{definition:NoetherianGliders}
We write $\prefrag FR, \preglid FR,$ and $\glid FR$ for the full subcategories of $\Prefrag FR,\allowbreak \Preglid FR,$ and $\Glid FR$, respectively, consisting of all Noetherian objects.
\end{definition}

\begin{corollary}
The categories $\prefrag FR$ and $\preglid FR$ are quasi-abelian.  The category $\glid FR$ is deflation quasi-abelian.
\end{corollary}

\begin{proof}
This follows from proposition \ref{proposition:PreglidIsQuasiAbelian} and theorem \ref{theorem:GlidIsDeflationQuasiAbelian}, together with theorem \ref{theorem:NoetherianObjects}.
\end{proof}

The following proposition helps recognizing Noetherian glider representations.

\begin{proposition}\label{proposition:WhenNoetherian}\label{proposition:NoetherianObjects}
Let $\Gamma = \bZ$ and $\Lambda = \bZ^{\leq 0}.$  Let $FR$ be a nonnegative filtration of a ring $R.$  Let $M \in \Glid FR$ be a glider representation.  The following are equivalent.
\begin{enumerate}
	\item $M$ is Noetherian in $\Glid FR,$
	\item $M$ is Noetherian in $\Prefrag FR,$
	\item $M$ is Noetherian in $\Mod \FF_\Lambda R,$
	\item the $F_0 R$-module $\oplus_{\lambda \in \Lambda} M(\lambda)$ is Noetherian,
	\item the $F_0 R$-module $M(0)$ is Noetherian and $M(\lambda) = 0$ for $\lambda \ll 0$. 
\end{enumerate}
\end{proposition}

\begin{proof}
As each of the subcategories $\Glid FR \subseteq \Prefrag FR \subseteq \Mod \FF_\Lambda R$ is closed under subobjects (see proposition \ref{proposition:PreglidIsTorsionfree} and corollary \ref{corollary:GlidersClosedUnderSubobjects}), it is clear that the first three statements are equivalent.  The other statements are easy.
\end{proof}

\begin{example}
Let $\Gamma, \Lambda$ be as in proposition \ref{proposition:WhenNoetherian}.  Let $FR$ be the given by
\[F_i R = \begin{cases} 0 & i < 0, \\ \bC & i = 0, \\ \bC[t] & i > 0.\end{cases}\]
The projective prefragment $\FF_\Lambda R(0,-)$ is Noetherian, but the projective fragment $\FF_\Lambda R(-1,-)$ is not Noetherian (as $\FF_\Lambda R(-1,0) \cong \bC[t]$ is not a finitely generated $\bC$-vector space).

The preglider $L(\FF_\Lambda R(0,-)) = \oFF_\Lambda R(0,-)$ is not Noetherian.
\end{example}

%% file: Naturalgliders.tex
\section{Natural gliders and pregliders}\label{section:NaturalGliders}

Let $\Gamma$ be an ordered group and let $\Lambda \subseteq \Gamma$ be any subset.  Let $FR$ be a $\Gamma$-filtered ring.  In this section, we assume that $\Lambda$ has a maximal element which we denote by $0$.  In order to discuss the concept of a natural glider from \cite{CaenepeelVanOystaeyenBook19}, we introduce the category $\NPreglid FR$ of natural pregliders over $FR$.

\subsection{Natural pregliders}

As in \S\ref{subsection:RepresentationsOfCategories}, the inclusion $n\colon \{0, \infty\} \to \Lambda \coprod \{\infty\}$ induces a restriction functor $n^*\colon \Mod \oFF_\Lambda R \to \Mod \oFF_{\{0\}} R$, which has a left adjoint $n_!$ and a right adjoint $n_*.$  These adjoints are fully faithful (see \cite[remark~2.3 and example~2.13]{Psaroudakis14}).  The adjoint triple $(n_!, n^*, n_*)$ restricts to the corresponding subcategories of pregliders, as can be seen from the following explicit formulation.

\begin{proposition}\label{proposition:AdjointsNaturalPregliders}
The natural functor $n^*\colon \Preglid FR \to \operatorname{Preglid}_{\{0\}} FR$ has a left adjoint given by
\begin{align*}
n_!(M)(\lambda) &= \begin{cases} M(\lambda) & \lambda \in \{0, \infty\} \\ 0 & \mbox{otherwise,} \end{cases}
\intertext{and a right adjoint given by}
n_*(M)(\lambda) &= \{m \in M(\infty) \mid F_{\lambda^{-1}}R \cdot m \subseteq \im M(1_{0,\infty})\}.
\end{align*}
\end{proposition}

\begin{definition}
We write $\NPreglid FR$ for the essential image of the functor $n_*\colon \operatorname{Preglid}_{\{0\}} FR \to \Preglid FR.$  An object of $\NPreglid FR$ is called a \emph{natural preglider}.  We write $\nu\colon \NPreglid FR \to \Preglid FR$ for the embedding and $\rho\colon \Preglid FR \to \NPreglid FR$ for the corresponding left adjoint.
\end{definition}

\begin{remark}\label{remark:NaturalUnitEpiMono}
For each $M \in \Preglid FR,$ the unit of the adjunction $M \to \nu \circ \rho(M)$ is a monomorphism (as, for every $\lambda \in \Lambda$, it is given by the inclusion $M(1_{\lambda, \infty})\colon M(\lambda) \to M^*(\lambda)$) and an epimorphism (as the map $M(\infty) \to (\nu \circ \rho(M)) (\infty)$ is an isomorphism).
\end{remark}

\begin{remark}
Since $\NPreglid FR \simeq \operatorname{Preglid}_{\{0\}} FR$, we know that $\NPreglid FR$ is quasi-abelian.  However, as the following example shows, the embedding $\NPreglid FR \to \Preglid FR$ need not commute with colimits.
\end{remark}

\begin{example}\label{example:CokernelNaturalPregliders}
We set $\Gamma = \bZ$ and $\Lambda = \bZ^{\leq 0}.$  Let $R = \bC[t]$ with filtration $\bC \subseteq \bC[t] \subseteq \bC[t] \subseteq \ldots$ where $F_0 R = \bC.$  We consider the following map between natural pregliders:
\[
\xymatrix{
M \ar[d]^{f} & \cdots \ar@{^{(}->}[r]& 0 \ar@{^{(}->}[r]\ar[d]& \bC \ar@{^{(}->}[r]\ar[d] & \bC[t]\ar[d] \\
N & \cdots \ar@{^{(}->}[r] & 0\ar@{^{(}->}[r] & \bC[t^{\leq 1}] \ar@{^{(}->}[r] & \bC[t] \\
}\]
given by multiplication by $t$.  One readily verifies that $f$ is an inflation in $\Preglid FR$, and that the cokernel (in $\Preglid FR$) is not a natural preglider.
\end{example}

\begin{proposition}\label{proposition:InflationsToNaturalPregliders}
Let $s\colon K \inflation M$ be an inflation in $\Preglid FR$.  If $M$ is a natural preglider, then so is $K.$
\end{proposition}

\begin{proof}
Using that $M$ is a natural preglider, the adjunction gives the following commutative diagram:
\[\xymatrix{
K \ar@{>->}[r]^{s} \ar[d]& M \ar@{=}[d] \\
\nu \circ \rho(K) \ar[r] & M}\]
As $\Preglid FR$ is a quasi-abelian category, we know that the canonical morphism $K \to \nu \circ \rho(K)$ is an inflation.  As it is also an epimorphism (see remark \ref{remark:NaturalUnitEpiMono}), we find that it is an isomorphism.  Hence, $K$ is a natural preglider.
\end{proof}

\begin{proposition}\label{proposition:NPreglidExtensionClosed}
The subcategory $\NPreglid FR$ lies extension-closed in $\Preglid FR.$
\end{proposition}

\begin{proof}
Let $X \inflation Y \deflation Z$ be a conflation in $\Preglid FR$.  Assume that $X$ and $Z$ are natural pregliders.  Applying $\nu \circ \rho$ yields the commutative diagram
\[\xymatrix{
X \ar@{=}[d]\ar@{>->}[rr] && Y \ar@{->>}[rr] \ar[d] && Z\ar@{=}[d] \\
X \ar[rr]^{i} && \nu \circ \rho(Y) \ar[rr]^{p} && Z }\]
As $\nu \circ \rho$ is left exact, we know that $i = \ker(p)$.  Furthermore, as the composition $Y \to \nu \circ \rho(Y) \to Z$ is a deflation, and $\Preglid FR$ is quasi-abelian, we find that $p\colon \nu \circ \rho(Y) \to Z$ is a deflation.  Hence, the bottom row is a conflation.  The short five lemma now shows that the canonical morphism $Y \to \nu \circ \rho(Y)$ is an isomorphism.  Hence, $Y$ is a natural preglider.
\end{proof}

\begin{remark}\label{remark:TwoExactStructuresOnNPreglid}
The above puts forward two different conflation structures on $\NPreglid FR.$
\begin{enumerate}
	\item As $\NPreglid FR \simeq \operatorname{Preglid}_{\{0\}} FR$, we know that $\NPreglid FR$ is a quasi-abelian category.  The conflation structure is induced by $\Preglid FR$ via the embedding $n_!\colon \operatorname{Preglid}_{\{0\}} FR \to \Preglid FR.$
	\item By proposition \ref{proposition:NPreglidExtensionClosed}, we know that the exact structure of $\Preglid FR$ induces an exact structure on $\NPreglid FR$: the conflations in $\NPreglid FR$ are those sequences which are conflations in $\Preglid FR.$
\end{enumerate}
In general, these two conflation structures need not coincide.  As an example, the morphism $f$ in example \ref{example:CokernelNaturalPregliders} is an inflation in the first exact structure (as it is the kernel of its cokernel), but not in the second exact structure.
\end{remark}

\subsection{Natural gliders}

Having discussed the category of natural pregliders, we now turn to the category of natural gliders.

\begin{definition}
A glider $M \in \Glid FR$ is called a \emph{natural glider} if it is isomorphic to $Q(\bar{M})$ for a natural preglider $\bar{M}$.  We write $\NGlid FR$ for the full subcategory of $\Glid FR$ given by the natural gliders.
\end{definition}

\begin{remark}
The category $\NGlid FR$ is the essential image of $\NPreglid FR$ under the localization functor $Q\colon \Preglid FR \to \Glid FR.$
\end{remark}

In proposition \ref{proposition:CriterionNatural}, we establish a recognition result for natural gliders.  We start with the following lemma.

\begin{lemma}\label{lemma:NaturalPregliders}
Let $s\colon M \stackrel{\sim}{\rightarrow} N$ be a weak isomorphism in $\Preglid FR$ (i.e. $s \in \Sigma$).  If $N$ is a natural preglider, then so is $M.$
\end{lemma}

\begin{proof}
As $s\colon M \to N$ is a weak isomorphism, it factors as $M \stackrel{\sim}{\deflation} I \stackrel{\sim}{\inflation} N$.  It follows from proposition \ref{proposition:InflationsToNaturalPregliders} that $I$ is a natural preglider, and then from proposition \ref{proposition:NPreglidExtensionClosed} that $M$ is a natural preglider (the last statement uses that $\ker(M \stackrel{\sim}{\deflation} I)$ is a natural preglider).
\end{proof}

\begin{proposition}\label{proposition:CriterionNatural}
A glider $M \in \Glid FR$ is natural if and only if $L(M) \in \Preglid FR$ is a natural preglider.
\end{proposition}

\begin{proof}
If $L(M)$ is a natural preglider, then $QL(M) \cong M$ is a natural glider.  For the reverse implication, let $M \in \Glid FR$ be a natural glider.  We know that there is a natural preglider $\bar{M} \in \Preglid FR$ such that $Q(\bar{M}) \cong M.$  By adjointness, we find a weak isomorphism $L(M) \stackrel{\sim}{\rightarrow} \bar{M}.$  Lemma \ref{lemma:NaturalPregliders} shows that $L(M)$ is a natural preglider.
\end{proof}

The following is an analogon of \ref{proposition:InflationsToNaturalPregliders} (see also \cite[lemma 1.6.2(3)]{CaenepeelVanOystaeyenBook19}).

\begin{proposition}\label{proposition:InflationsToNaturalGliders}
Let $s\colon K \inflation M$ be an inflation in $\Glid FR$.  If $M$ is a natural glider, then so is $K.$
\end{proposition}

\begin{proof}
Let $K \stackrel{i}{\inflation} M \stackrel{p}{\deflation} N$ be a conflation in $\Glid FR$ with $M \in \NGlid FR.$  Applying the functor $L\colon \Glid FR \to \Preglid FR$, we find a sequence $LK \stackrel{Li}{\to} LM \stackrel{Lp}{\deflation} LN$.  The map $Li$ factors as $LK \stackrel{\sim}{\rightarrow} \ker Lp \inflation LM$, where the first map is a weak isomorphism.  It now follows from propositions \ref{proposition:InflationsToNaturalPregliders} and \ref{proposition:CriterionNatural} that $LK \in \NPreglid FR,$ and hence $K \in \NGlid FR,$ as required.
\end{proof}

\begin{corollary}
The conflation structure on $\NGlid FR$ induced by the embedding $\NGlid FR \to \Glid FR$ gives $\NGlid FR$ the structure of a deflation-exact category.
\end{corollary}

The following diagram can be appended to the diagram given in figure \ref{figure:TheDiagram}.

\begin{theorem}\label{theorem:NaturalGliders}
The following diagram commutes:
\[\xymatrix{
{\Preglid FR} && {\Glid FR} \ar[ll]_{L} \\
{\NPreglid FR} \ar[u]^{\nu} && {\NGlid FR} \ar[ll]_{L} \ar[u]^{\sigma} }\]
Moreover, the functors are fully faithful, the vertical functors have left adjoints, and the horizontal functors have right adjoints.
\end{theorem}

\begin{proof}
That the functor $L\colon \Glid FR \to \Preglid FR$ restricts to the natural gliders and pregliders, has been shown in proposition \ref{proposition:CriterionNatural}.  This shows that the diagram commutes.  By definition of the category $\NGlid FR$, the functor $Q\colon \Preglid FR \to \Glid FR$ also restricts to natural gliders and pregliders.  This shows that the horizontal functors have right adjoints.

It follows from proposition \ref{proposition:AdjointsNaturalPregliders} that $\nu$ has a left adjoint $\rho\colon \Preglid FR \to \NPreglid FR.$  It follows from proposition \ref{proposition:CycleOfAdjoints}\eqref{enumerate:CycleOfAdjointsLocOuter} (with $F = \rho \circ L$, $H = \sigma$, and $G = Q$) that $Q \circ \rho \circ L$ is left adjoint to $\sigma.$
\end{proof}

\begin{corollary}
The category $\NGlid FR$ is complete and cocomplete. 
\end{corollary}

\begin{proof}
It follows from theorem \ref{theorem:NaturalGliders} that $\NGlid FR$ is a reflective subcategory of $\Glid FR$, as well as a coreflective subcategory of $\NPreglid FR$, both of which are complete and cocomplete.
\end{proof}

\begin{remark}
As in proposition \ref{proposition:AdjointsNaturalPregliders}, the category $\operatorname{Preglid}_{\{0\}} FR$ admits two fully faithful embeddings into $\Preglid FR$: via $n_!$ and via $n_*$.  Both essential images contain the subcategory $i_*(\Mod R).$

Since $n_!(\operatorname{Preglid}_{\{0\}} FR)$ is a Serre subcategory of $\Preglid FR$ (and $i_*(\Mod R)$ is deflation-percolating in $\Preglid FR$, see proposition \ref{proposition:GlidAsLocalization}), we find that $i_*(\Mod R)$ is a deflation-percolating subcategory of $n_!(\operatorname{Preglid}_{\{0\}} FR)$.  The quotient $n_!(\operatorname{Preglid}_{\{0\}} FR) / i_*(\Mod R)$ is equivalent to $\operatorname{Glid}_{\{0\}} FR \simeq \Mod S.$

It follows from proposition \ref{proposition:InflationsToNaturalPregliders} that $i_*(\Mod R)$ is a deflation-percolating subcategory of $\NPreglid FR.$  It follows from lemma \ref{lemma:NaturalPregliders} that $\NPreglid FR / i_*(\Mod R)$ is equivalent to $\NGlid FR$ as conflation categories. 

Note that $n_!(\operatorname{Preglid}_{\{0\}} FR) \simeq \NPreglid FR$ as categories, but not as conflation categories.
\end{remark}

%% file: DerivedRecollement.tex
\section{The derived category of glider representations as a Verdier localization}\label{TheDerivedCategoryOfGliderRepresentationsAsAVerdierLocalization}

In \S\ref{section:GliderCategory}, we introduced the category of glider representations as a localization of the category of pregliders.  In \S\ref{section:PropertiesOfTheCategoryOfGliderRepresentations}, we showed that $\Glid FR$ can be seen as the quotient $\Preglid FR / i_*(\Mod R),$ giving $\Glid FR$ the structure of a deflation-exact category.  In this section, we show that this quotient induces a Verdier localization sequence of the derived categories.

\subsection{Projective gliders}\label{subsection:ProjectiveGliders} We start by recalling the definition of a projective object in a deflation-exact category (see, for example, \cite{BazzoniCrivei13, Buhler10}).

\begin{definition}\label{definition:ProjectiveInDeflationExactCat}
	Let $\EE$ be a deflation-exact category.
	\begin{enumerate}
		\item An object $P\in \EE$ is called \emph{projective} if $\Hom(P,-)\colon \EE \to \Ab$ is an exact functor.  We say that $\EE$ has \emph{enough projectives} if for every object $M$ of $\EE$ there is a deflation $P\twoheadrightarrow M$ where $P$ is projective.
		\item Dually, an object $I$ is called \emph{injective} if $\Hom(-,I)\colon \EE^\circ \to \Ab$ is an exact functor.   We say that $\EE$ has \emph{enough injectives} if for every object $M$ of $\EE$ there is an inflation $M \inflation I$ where $I$ is injective.
	\end{enumerate}
	We write $\Proj \EE$ and $\Inj \EE$ for the full subcategories of projectives and injectives, respectively.
\end{definition}

The following proposition (see \cite[proposition~3.22]{HenrardVanRoosmalen19b}) characterizes projective modules in a deflation-exact category.

\begin{proposition}\label{proposition:ProjectiveCharacterizations}
Let $\EE$ be a deflation-exact category. The following are equivalent:
\begin{enumerate}
	\item $P$ is projective.
	\item\label{Item:ProjLift} For all deflations $f\colon X\twoheadrightarrow Y$ and any map $g\colon P\rightarrow Y$ there exists a map $h\colon P\rightarrow X$ such that $g=f\circ h$.
	\item\label{Item:ProjSplit} Any deflation $f\colon X\twoheadrightarrow P$ is a retraction, i.e.~there exist a map $g\colon P\rightarrow X$ such that $f\circ g=1_P$.
\end{enumerate}
\end{proposition}

\begin{proposition}\label{proposition:AboutProjectives}
\begin{enumerate}
	\item An object $M \in \Preglid FR$ is projective if and only if $\iota(M) \in \Mod \oFF_\Lambda R$ is projective.
	\item An object $M \in \Prefrag FR$ is projective if and only if $\eta(M) \in \Mod \FF_\Lambda R$ is projective.
\end{enumerate}
\end{proposition}

\begin{proof}
The functor $\iota\colon \Preglid FR \to \Mod \oFF_\Lambda R$ reflect conflations.  This shows that if an object $\iota(M) \in \Mod \oFF_\Lambda R$ is projective, then $M \in \Preglid FR$ is projective.  For the other direction, assume that $M \in \Preglid FR$ is projective.  Let $P \deflation \iota(M)$ be an epimorphism in $\Mod \oFF_\Lambda R$ where $P$ is projective.  As all projective objects in $\Mod \oFF_\Lambda R$ are pregliders (and projective in $\Preglid FR$), and $M$ is projective as a preglider, we find that $M$ is a direct sumand of $P$.  Hence, $\iota(M)$ is a direct summand of $P$.  We find that $\iota(M)$ is projective.

The other statement can be proven in an analogous fashion.
\end{proof}

\begin{proposition}\label{proposition:AboutProjectiveGliders}
\begin{enumerate}
	\item An object $M \in \Glid FR$ is projective if and only if $\phi(M) \in \Prefrag FR$ is projective.
	\item An object $M \in \Glid FR$ is projective if and only if $L(M) \in \Preglid FR$ is projective.
\end{enumerate}
\end{proposition}

\begin{proof}
The proof of the first statement is analogous to the proof of proposition \ref{proposition:AboutProjectives}.  For the second statement, as $L$ has a conflation-exact right adjoint, we find that $L$ maps projective objects to projective objects.  Indeed, for every projective $M \in \Glid FR$, we have that the functor $\Hom(LM,-) \cong \Hom(M,Q(-))$ is conflation-exact.

Assume now that $LM$ is projective in $\Preglid FR$.  Let $K \inflation N \deflation M$ be any conflation in $\Glid FR$, ending in $M.$  As $L$ commutes with cokernels and $\Preglid FR$ is quasi-abelian, we find that $LN \deflation LM$ is a deflation.  As $LM$ is projective, this deflation splits.  Applying the functor $Q$, we find that the given map $N \deflation M$ is a split deflation.  As any conflation ending in $M$ splits, we find that $M$ is projective. 
\end{proof}

\begin{corollary}\label{corollary:EnoughProjectives}
The categories $\Preglid FR, \Prefrag FR$ and $\Glid FR$ have enough projectives.
\end{corollary}

\begin{proof}
This follows from propositions \ref{proposition:AboutProjectives} and \ref{proposition:AboutProjectiveGliders}.
\end{proof}

\subsection{A recollement on the derived level}
We begin by observing that the upwards arrows in the diagram in figure \ref{figure:TheDiagram} induce derived equivalences.

\begin{proposition}\label{proposition:DerivedEquivalences}
The following functors are triangle equivalences:
\begin{enumerate}
	\item $\iota\colon \Db(\Preglid FR) \to \Db( \Mod \oFF_\Lambda R)$,
	\item $\eta\colon \Db(\Prefrag FR) \to \Db(\Mod \FF_\Lambda R)$,
	\item $\phi\colon \Db(\Glid FR) \to \Db(\Prefrag FR)$.
\end{enumerate}
\end{proposition}

\begin{proof}
It is shown in corollary \ref{corollary:EnoughProjectives} that $\Preglid FR$ has enough projectives.  It follows from proposition \ref{proposition:AboutProjectives} that the functor $\iota\colon \Preglid FR \to \Mod (\oFF_\Lambda R)$ induces an equivalence between the categories of projective objects, and hence between the homotopy categories of projectives.  The result now follows from \cite{HenrardVanRoosmalen19b}.  The other proofs are similar. 
\end{proof}

\begin{proposition}
	The recollement 
	\[\xymatrix{
{\Mod R} \ar[rr]|{i_*} && {\Mod \oFF_\Lambda R} \ar[rr]|{j^*} \ar@/^/[ll]^{i^!} \ar@/_/[ll]_{i^*} && {\Mod \FF_\Lambda R} \ar@/^/[ll]^{j_*} \ar@/_/[ll]_{j_!}
}\]
lifts to a recollement on the bounded derived categories.
\end{proposition}

\begin{proof}
	As the categories $\Mod(R), \Mod(\overline{\FF}_\Lambda R)$ and $\Mod(\FF_\Lambda R)$ have enough injective and projective objects, the six functors $i^*,i_*,i^!,j_!,j^*$ and $j_*$ lift to triangle functors on the bounded derived categories. It now follows from \cite[theorem~7.2; lemma~7.3]{Psaroudakis14} that we obtain a recollement
\[\xymatrix{
		\Db(\Mod(R))\ar[rr]|{i_*} && \Db(\Mod(\overline{\FF}_\Lambda R)) \ar@/^/[ll]^{R i^!}\ar@/_/[ll]_{i^*} \ar[rr]|{j^*} && \Db(\Mod(\FF_\Lambda R))\ar@/^/[ll]^{j_*}\ar@/_/[ll]_{L j_!}.
	}\] of triangulated categories.
\end{proof}

	Note that the sequence $\Db(\Mod(R))\xrightarrow{i_*}\Db(\Mod(\overline{\FF}_\Lambda R)) \xrightarrow{j^*} \Db(\Mod(\FF_\Lambda R))$ is a Verdier localization sequence. By proposition \ref{proposition:GlidAsLocalization} and theorem \ref{theorem:MainTheoremDerivedCatOfLoc}, the sequence 
	\[\DModRb(\Preglid FR)\to \Db(\Preglid FR)\xrightarrow{Q} \Db(\Glid FR)\]
	is a Verdier localization sequence as well. 
	
\begin{proposition}\label{proposition:CondC2}
The natural functor $\Db(\Mod(R)) \to \DModRb(\Preglid FR)$ is a triangle equivalence.
\end{proposition}
	
\begin{proof}
By \cite[proposition~5.9]{HenrardVanRoosmalen19b} it suffices to show that the following property holds: for each conflation $E'\inflation E \deflation A$ in $\Preglid FR$ with $A\in \Mod(R)\subseteq \Preglid FR$ there exists a commutative diagram
	\[\xymatrix{
		A''\ar@{>->}[r]\ar[d] & A'\ar@{->>}[r]\ar[d] & A\ar@{=}[d]\\
		E'\ar@{>->}[r] & E\ar@{->>}[r] & A
	}\] where the rows are conflations and the top row belongs to $\Mod(R)$. 
	
	Now let $E'\inflation E\deflation A$ be a conflation with $A\in \Mod(R)$. As $\Mod(R)$ has enough projectives, there exists a conflation $K\inflation P\deflation A$ with $P$ a projective in $\Mod(R)$.  Since $P$ is also projective in $\Preglid FR,$ we obtain a commutative diagram
		\[\xymatrix{
		K\ar@{>->}[r]\ar@{.>}[d] & P\ar@{->>}[r]\ar@{.>}[d] & A\ar@{=}[d]\\
		E'\ar@{>->}[r] & E\ar@{->>}[r] & A
	}\] where the map $P\to E$ is obtained by the lifting property of the projective $P$. Clearly $K\in \Mod(R)$ as well. This shows the claim.
\end{proof}

The following result summarizes the results of this section.

\begin{proposition}\label{proposition:VerdierLocalization}
There is a Verdier localization sequence
\[\Db(\Mod R)\xrightarrow{i^*} \Db(\Preglid FR)\xrightarrow{Q} \Db(\Glid FR),\]
which can be completed to a recollement of triangulated categories.
\end{proposition}

%% file: BialgebraFiltration.tex
\section{Glider representations of bialgebra filtrations}\label{section:HopfCategories}

In this section, we consider the category of glider representations associated to bialgebras.  The definitions are based on the glider representations of groups \cite{CaenepeelVanOystaeyen18, CaenepeelVanOystaeyenBook19}.  Our aims in this section are to related the monoidal category of glider representations to the notion of semi-Hopf categories (see \S\ref{susection:SemiHopf}), and to show that one can, in general, recover the original bialgebra from the monoidal category of glider representations.

Thus, let $k$ be a field, $\Gamma$ be an ordered group, $B$ a $k$-bialgebra and $FB$ a $\Gamma$-filtration of $B$ by bialgebras, i.e.~a collection of subbialgebras $(FB_\gamma)_{\gamma\in \Gamma^+}$ of $B$ satisfying 
\begin{enumerate}
	\item $FB_{\gamma}\cdot FB_{\gamma'}\subseteq FB_{\gamma \gamma'}$,
	\item $\Delta(FB_{\gamma})\subseteq FB_{\gamma}\otimes FB_{\gamma}$.
\end{enumerate}	

\begin{example}
Consider a group $G$ and a finite group filtration
	\[1\subseteq G_1\subseteq G_2 \subseteq \dots \subseteq G_n=G.\]
	From this filtration, one obtains the associated filtration of group algebras over some field $k$	
	\[k\subseteq k[G_1]\subseteq k[G_2]\subseteq \dots \subseteq k[G_n]=k[G].\]
	This is a filtration of $k[G]$ by bialgebras as above.
\end{example}

\subsection{The monoidal structure and the connection with semi-Hopf categories}  With notations as introduced before, we show that the filtered companion category $\FF_\Lambda B$ is a semi-Hopf category and that the category $\Glid FB$ of glider representations is a full monoidal subcategory of $\Mod \FF_\Lambda B$.  Moreover, the natural embedding $\Glid FB \to \Mod \FF_\Lambda$ lifts to an equivalence on the level of derived categories.

\begin{proposition}\label{proposition:CompanionCategorySemiHopf}
The categories $\FF_\Lambda B$ and $\oFF_\Lambda B$ have the structure of $k$-linear semi-Hopf categories.
\end{proposition}

\begin{proof}
	We only show that $\FF_\Lambda B$ is a $k$-linear semi-Hopf category, the case $\oFF_\Lambda B$ is similar. Let $\lambda_1,\lambda_2,\lambda_3\in \Lambda$ and set 
	\[\circ_{\lambda_1,\lambda_2,\lambda_3}\colon FB_{\lambda_3\lambda_2^{-1}}\otimes FB_{\lambda_2\lambda_1^{-1}}\to FB_{\lambda_3\lambda_1^{-1}}\colon a\otimes b\mapsto ab.\] The map $\circ_{\lambda_1,\lambda_2,\lambda_3}$ is well-defined as $\left(FB_{\gamma}\right)_{\gamma\in \Gamma^+}$ is an algebra filtration. 
	Moreover, as $B$ is a bialgebra, the multiplication map $\circ_{\lambda_1,\lambda_2,\lambda_3}$ is a coalgebra map. It follows that $\FF_\Lambda B$ is enriched over $\underline{\CC}(\MM_k)$ and thus $\FF_\Lambda B$ is a $k$-linear semi-Hopf category.
\end{proof}

Let $\Mod_k(\FF_\Lambda B)$ and $\Mod_k(\oFF_\Lambda B)$ be the categories of covariant additive functors to $\Vect(k)$. Define a tensor product $\otimes$ on $\Mod _k(\oFF_\Lambda B)$ by setting $(M\otimes N)(\lambda)=M(\lambda)\otimes_k N(\lambda)$ for each $\lambda\in \Lambda\coprod \left\{\infty\right\}$.

\begin{proposition}\label{proposition:ExactTensor}
The categories $\Mod_k(\FF_\Lambda B)$ and $\Mod_k(\oFF_\Lambda B)$ are $k$-linear monoidal categories. Moreover, the tensor product is exact.
\end{proposition}

\begin{proof}
	It follows from \cite[proposition~3.2]{BatistaCaenepeelVercruysse16} that $\Mod_k(\FF_\Lambda B)$ and $\Mod_k(\oFF_\Lambda B)$ are $k$-linear tensor categories. The exactness of the tensor product is straightforward to verify.
\end{proof}

The exactness of the tensor product over $k$ immediately gives the following corollary.

\begin{corollary}
The categories $\Preglid FB$ and $\Prefrag FB$ inherit a monoidal structure from $\Mod_k(\oFF_\Lambda B)$ and $\Mod_k(\FF_\Lambda B)$, respectively.
\end{corollary}

\begin{proposition}\label{proposition:MonoidalLocalization}
\begin{enumerate}
	\item The subcategory $i_*(\Mod_k(B)) \subseteq \Mod_k(\FF_\Lambda B)$ is a tensor ideal.
	\item The subcategory $i_*(\Mod_k(B)) \subseteq \Preglid FB$ is a tensor ideal.
	\item The category $\Glid FB$ is a monoidal category, and the quotient functor $Q\colon \Preglid FR \to \Glid FR$ is universal in the category of monoidal categories and functors.
\end{enumerate}
\end{proposition}

\begin{proof}
	The first two statements are trivial.  The last statement follows from the second one (see \cite[corollary 1.4]{Day73}).
\end{proof}

\begin{remark}
The monoidal structure on $\Glid FB = \Sigma^{-1}\Preglid FR$ (see definition \ref{definition:GlidAsLoc}) can be described as follows.  As $\Ob (\Glid FR) = \Ob(\Preglid FR)$, the tensor product on gliders is the same as the tensor product on pregliders.  Let now $M \stackrel{\sim}{\leftarrow} M' \to X$ and $N \stackrel{\sim}{\leftarrow} N' \to Y$ be morphisms in $\Glid FR$.  The corresponding morphism from $M \otimes N$ to $X \otimes Y$ is given by:
\[M \otimes N \stackrel{\sim}{\leftarrow} M' \otimes N' \to X \otimes Y.\]
Here, we use that if $s,t\in \Sigma$, then $s \otimes t \in \Sigma.$

As $Q\colon \Glid FB \to \Preglid FB$ is the identity on objects, we can give $Q$ the structure of a monoidal functor by choosing the identity map for $J_{M,N}\colon Q(M) \otimes Q(N) \to Q(M \otimes N)$.
\end{remark}

\begin{remark}
The tensor product on $\Glid FB$ is conflation-exact; but need not commute with colimits.  As an example, the cokernel of the morphism $f\colon M \to N$ in $\Glid FR$ from example \ref{example:SymmetricGroup1} is zero.  Let $T$ be the prefragment given by $0 \subseteq \mathbb{C}$.  Note that $T$ is a glider with $T(\infty) = \mathbb{C}S_3$.  The cokernel of the map $f \otimes T\colon M \otimes T \to N \otimes T$ is isomorphic to $T.$

In particular, the monoidal structure on $\Glid FR$ is not closed (i.e. the tensor product does not have a right adjoint).
\end{remark}

\begin{proposition}
The functor $\phi\colon \Glid FB \to \Preglid FB$ is a monoidal functor.
\end{proposition}

\begin{proof}
Directly from proposition \ref{proposition:MonoidalLocalization}.
\end{proof}

\begin{remark}
The category $\Glid FR$ is a full subcategory of $\Preglid FR$ via the functor $\phi$.  There is a unique structure of a monoidal functor on $\phi$ such that $\phi \circ Q = j_*$ (this follows from proposition \ref{proposition:MonoidalLocalization}); the morphisms $J_{M,N}\colon \phi (M) \otimes \phi (N) \to \phi(M \otimes N)$ are given by the identities.
\end{remark}

Combining the results of this section, we obtain the following theorem.

\begin{theorem}
Let $B$ a $k$-bialgebra and $FB$ a $\Gamma$-filtration of $B$ by bialgebras.  Let $\Lambda \subseteq \Gamma$ be any subset.  The filtered companion category $\FF_\Lambda B$ has the structure of a semi-Hopf category and there is a monoidal triangle equivalence $\Db(\Glid FB) \to \Db(\Mod \FF_\Lambda B).$
\end{theorem}

\begin{proof}
It was verified in proposition \ref{proposition:CompanionCategorySemiHopf} that $\FF_\Lambda B$ has the structure of a semi-Hopf category.  The monoidal equivalence $\Db(\Glid FB) \to \Db(\Mod \FF_\Lambda B)$ is induced by the monoidal functor $\eta\circ \phi\colon \Glid FB \to \Mod \FF_\Lambda B.$
\end{proof}

\subsection{Glider representations and isocategorical groups}\label{subsection:Isocategorical}

Let $k$ be a field and let $G$ be a group.  It is known that one cannot recover the group $G$ from the monoidal category $\Mod kG$ alone.  Indeed, we say that the groups $G$ and $H$ are \emph{isocategorical} over $k$ if $\Mod kG$ and $\Mod kH$ are monoidally equivalent.  Examples of nonisomorphic isocategorical groups are given in \cite[Theorem 1.2]{EtingofGelaki01}.

In \cite{CaenepeelVanOystaeyen19b}, the authors started from the $\bZ$-filtered algebra $F(kG)$ given by
\[ F_i(kG) = \begin{cases} 0 & i < 0, \\ k \cdot e_G & i = 0, \\ kG & i \geq 1, \end{cases}\]
and let $\Lambda = \{0,1\}$.  Similar to the study of the characters of a group, one can look at various decategorifications of the category $\glid F(kG)$ of glider representations, such as the glider character ring or the (reduced) representation ring \cite{CaenepeelJanssens19, CaenepeelVanOystaeyenBook19}.  Here, it can be shown that these invariants are sufficient to distinguish between the groups $Q_8$ and $D_8$ (\cite{CaenepeelVanOystaeyen19b, CaenepeelVanOystaeyenBook19}) and even between some isocategorical groups (\cite{CaenepeelJanssens19}).  As of yet, there is no example known of groups $G$ and $H$ that cannot be distinguished using the glider character ring and the (reduced) representation ring.

In this section, we approach a similar question from a different perspective: can one recover the group $G$ from the monoidal category $\Glid F(kG)$?  As the glider character ring and the (reduced) representation ring only depend on the monoidal category $\Glid F(kG),$ the information one can recover from these rings is bounded above by the information one can recover from the category $\Glid F(kG).$  In corollary \ref{corollary:RecoverGroupsFromGliders} below, we show that the monoidal category $\Glid F(kG)$ alone is sufficient to reconstruct the group $G.$

We will proceed with more generality.  Let $B$ be a finite-dimensional $k$-bialgebra, and consider the bialgebra $\bZ$-filtered algebra $FB$ given by
\begin{equation}\label{eq:StandardOneStep}
F_i B = \begin{cases} 0 & i < 0, \\ k \cdot 1_B & i = 0, \\ B & i \geq 1. \end{cases}
\end{equation}
We refer to such a filtration as the \emph{standard one-step filtration} of $B.$  We write $\mod B$ for the category of Noetherian (or, equivalently, finite-dimensional) modules.

For ease of notation, we consider the category $\glid FR$ of Noetherian glider representations as a subcategory of $\Prefrag FR$ via the embedding $\phi\colon \Glid FR \to \Prefrag FR$, i.e.~we will describe an object of $\glid FR$ as an $\FF_\Lambda R$-module.

As in definition \ref{definition:NoetherianGliders}, we write $\glid FB$ for the subcategory of noetherian objects in $\Glid FB$.  By proposition \ref{proposition:NoetherianObjects}, these are exactly the glider representation with finite total dimension.  In particular, as the bialgebra $B$ is finite-dimensional, $P_0 = \FF_\Lambda R(0,-)$ and $P_{-1} = \FF_\Lambda R(-1,-)$ lie in $\glid FR$.

Our goal is to recover the bialgebra $B$ from the monoidal category $\glid FB.$  We provide a short overview.  First, we define full monoidal subcategories $\MM$ and $\VV$ of $\glid FB.$  Here, the category $\MM$ can be thought of as consisting of all gliders $M \in \glid FB$ for which $M(1_{-1,0})\colon M(-1) \to M(0)$ is an isomorphism; the category $\VV$ consists of all $M \in \glid FB$ for which $M(-1) = 0$ (and hence, $\VV \simeq \Vect_k$).  The embedding $\VV \to \glid FB$ has a monoidal left adjoint.  We then obtain a fiber functor $\MM \to \glid FB \to \VV \simeq \Vect_k.$  Finally, we show that the reconstruction theorem of finite-dimensional bialgebras applied to this fiber functor yields the bialgebra $B.$

We will be careful to construct the categories $\VV, \MM$ and the fiber functor $\MM \to \VV$ using only categorical properties of $\glid FR$ (thus, without referring to the internal structure of the objects of $\glid FR$).

We start with specifying the full subcategories $\MM$ and $\VV$ of $\glid FB.$

\begin{notation}\label{notation:RecoveringBialgebra}
The category $\glid FB$ has precisely two isomorphism classes of indecomposable projective objects: $P_{-1}$ and $P_0$.  They are distinguished by the property that $\Hom(P_{-1}, P_0) = 0$ while $\Hom(P_0, P_{-1}) \not= 0$.

We consider the subcategory $\MM$ of $\glid FB$ given by those objects $M$ for which $\Hom(P_0,M) \cong \Hom(P_{-1},M)$ as vector spaces.

We write $\VV$ for the full subcategory of $\glid FB$ consisting of those objects $M$ for which $\Hom(P_{-1},M) = 0.$ 
\end{notation}

\begin{remark}
\begin{enumerate}
	\item Note that $P_0 \in \VV$.
	\item For any $M \in \glid FR,$ we have $\dim_k \Hom(P_0,M) < \infty$ (see proposition \ref{proposition:NoetherianObjects}).  As \[M(1_{-1,0})\colon M(-1) \to M(0)\] is a monomorphism, we find that $M \in \MM$ if and only if $M(1_{-1,0})$ is an isomorphism. 
\end{enumerate}
\end{remark}

\begin{proposition}
The categories $\MM$ and $\VV$ are monoidal subcategories of $\glid FB.$ 
\end{proposition}

\begin{proof}
Directly from the definitions.
\end{proof}

\begin{proposition}
The monoidal embedding $\VV \to \glid FB$ has a monoidal right adjoint $R\colon \glid FB \to \VV$, which is unique up to monoidal natural equivalence.
\end{proposition}

\begin{proof}
The right adjoint is given by mapping an object $M \in \glid FB$ to the functor
\[R(M)(i) = \begin{cases} 0 & i = -1, \\ M(0) & i=0.\end{cases}\]
It is clear that this is a monoidal functor.  As $R$ is an adjoint to the embedding in the 2-category of monoidal categories, monoidal functors, and monoidal natural transformations, $R$ is unique up to monoidal natural equivalence.
\end{proof}

\begin{lemma}\label{lemma:EquivalenceWithVect}
\begin{enumerate}
	\item A monoidal equivalence $E\colon \Vect_k \to \Vect_k$ is monoidally naturally isomorphic to the identity.
	\item Let $H_1, H_2\colon \Vect_k \to \VV$ be monoidal functors.  If $H_1$ and $H_2$ are equivalences, then there is a monoidal natural isomorphism $H_1 \to H_2.$
\end{enumerate}
\end{lemma}

\begin{proof}
The first statement follows from the reconstruction theorem of finite-dimensional bialgebras, cf.~\cite[Theorem 5.2.3]{EtingofGelakiNikshychOstrik15}).

For the second part, let $H_1, H_2\colon \Vect_k \to \VV$ be as in the statement of the proposition, and let $H_1^{-1}$ be a monoidal quasi-inverse of $H_1.$  It is established in the first part that $H_1^{-1} \circ H_2 \cong 1_{\Vect}.$  This shows that $H_1 \cong H_2,$ as required.
\end{proof}

\begin{remark}
The composition $\Xi\colon \MM \to \glid FB \to \VV \to \Vect_k$ is a faithful monoidal functor.  In theorem \ref{theorem:RecoverBialgebra} below, we show that, by applying the reconstruction theorem of finite-dimensional bialgebras, we can recover the bialgebra $B$.
\end{remark}

For the next lemma, recall that a finite-dimensional $B$-module is a pair $(V,\rho_V)$ where $V$ is a finite-dimensional vector space and $\rho_V\colon B \to \End(V)$ is an algebra homomorphism.

\begin{lemma}\label{lemma:MMisModB}
There is a monoidal equivalence $\Phi\colon \mod B \to \MM$, mapping a $B$-module $M = (V, \rho_V)$ to the functor $\Phi(M)\colon \FF_\Lambda B \to \Vect_k$ given by $\Phi(M)(-1) = \Phi(M)(0) = V$ and, for each $b \in B = \Hom(-1,0)$, the corresponding map $\Phi(M)(b)\colon \Phi(M)(-1) \to \Phi(M)(0)$ is given by the action of $B$ on $V$ (i.e.~$\Phi(M)(b) = \rho(b)$).
\end{lemma}

\begin{proof}
It is straightforward to verify that this correspondence is a monoidal functor.  To see that it is an equivalence, we will construct a quasi-inverse.  Let $N \in \MM$.  As $\dim N(-1) = \dim N(0)$, the monomorphism $N(1_{-1,0})\colon N(-1) \to N(0)$ is an isomorphism.  As $N$ is a glider, there is a preglider $N' \in \Preglid FR$ such that $Q(N') \cong N.$  It is now easy to see that $N(-1) = N(0)$ is a $B$-submodule of $N'(\infty).$  The functor $\Psi\colon \MM \to \mod B$, given by mapping $N$ to the $B$-submodule $N(-1)$ is a quasi-inverse to $\Phi.$
\end{proof}

\begin{remark}
Note that the construction of the functor $\Phi$ in lemma \ref{lemma:MMisModB} is based on the forgetful functor $\mod B \to \Vect_k.$
\end{remark}

We now come to the main result of this subsection.

\begin{theorem}\label{theorem:RecoverBialgebra}
Let $\Gamma = \bZ$ and let $\Lambda = \{-1, 0\} \subset \bZ.$  Let $B$ and $B'$ be finite-dimensional $k$-algebras and let $FB$ and $FB'$ be their standard one-step filtrations.  The categories $\glid FB$ and $\glid FB'$ are monoidally equivalent if and only if $B \cong B'$ as bialgebras.
\end{theorem}

\begin{proof}
It is clear that an isomorphism $B \cong B'$ induces a monoidal equivalence $\glid FB$ and $\glid FB'.$  For the other direction, consider the subcategory $\VV$ of $\glid FB.$  There is a monoidal functor given by mapping $V \in \Vect_k$ to the glider representation $M_V$ given by $M_V(0) = V$ and $M_V(-1) = 0.$  It follows from lemma \ref{lemma:EquivalenceWithVect} that this monoidal functor is unique (up to monoidal natural isomorphism), and hence, does not depend on the bialgebra $B$.

The composition $\Xi\colon \MM \to \glid FB \to \VV \to \Vect_k$ is a faithful monoidal functor.  We claim that the bialgebra obtained from the reconstruction theorem of finite-dimensional bialgebras (\cite[Theorem 5.2.3]{EtingofGelakiNikshychOstrik15}) is $B$.  Consider the monoidal equivalence $\Phi\colon \mod B \to \MM$ from lemma \ref{lemma:MMisModB}.  It is straightforward to verify that the composition $\Phi \circ \Xi\colon \mod B \to \MM \to \glid FB \to \VV \to \Vect_k$ is the usual fiber functor.  It now follows from the reconstruction theorem of finite-dimensional bialgebras that one recovers the bialgebra $B$ from $\End(\Xi).$

As we have recovered the bialgebra $B$ from only the monoidal category $\glid FB$, we infer that $\glid FB$ and $\glid FB'$ are monoidally equivalent if and only if $B \cong B'$ as bialgebras.
\end{proof}

Restricting to the case where the bialgebra $B$ is a group algebra $kG$, we see that the monoidal category $\glid F(kG)$ of glider representations, associated with the trivial group filtration $\{e\} \subseteq G$ is sufficient to recover the group.

\begin{corollary}\label{corollary:RecoverGroupsFromGliders}
Let $G$ and $H$ be finite groups.  There is a group isomorphism $G \cong H$ if and only if there is a monoidal equivalence $\glid F(kG) \simeq \glid F(kH).$
\end{corollary}

\begin{remark}
The monoidal category $\glid F(kG)$ of Noetherian glider representations retains more information than the monoidal category $\mod kG$ of finite-dimensional modules.  In contrast, the Grothendieck ring of $\glid F(kG)$ is isomorphic to the product $\bZ \times \bZ$ (this follows from the derived equivalence $\Db(\glid F(kG)) \simeq \Db(\mod \FF_\Lambda(kG))$).  As such, the Grothendieck ring of $\glid F(kG)$ is disassociated from the group $G$. 
\end{remark}

As the following example illustrates, there is no direct generalization of corollary \ref{corollary:RecoverGroupsFromGliders} to the category of prefragments.

\begin{example}
Let $G$ and $H$ be any groups with the same number of elements.  Consider the standard one-step bialgebra filtrations, $F(kG)$ and $F(kH),$ of the group algebras $kG$ and $kH.$  Note that there is a coalgebra isomorphism $kG \to kH,$ mapping grouplike elements to grouplike elements (and hence mapping $1\in kG$ to $1 \in kH$), inducing an equivalence $\FF_\Lambda (kH) \to \FF_\Lambda (kG)$ of semi-Hopf categories.  We obtain a monoidal equivalence $\Prefrag F(kH) \to \Prefrag F(kG),$ even though $G$ and $H$ need not be isomorphic groups.
\end{example}

\subsection{Recovering an algebra from the category of glider representations}

The proof of theorem \ref{theorem:RecoverBialgebra} can easily be adapted to work in the setting of finite-dimensional algebras instead of bialgebras.

\begin{theorem}\label{theorem:RecoverAlgebra}
Let $\Gamma = \bZ$ and let $\Lambda = \{-1, 0\} \subset \bZ.$  Let $FA$ and $FA'$ be the standard one-step filtrations of finite-dimensional $k$-algebras $A$ and $A'$.  The categories $\glid FA$ and $\glid FA'$ are equivalent if and only if $A \cong A'$ as algebras.
\end{theorem}

\begin{proof}
The definition of the indecomposable projective objects, $P_0$ and $P_{-1}$, does not use the monoidal structure.  Similarly, the definition of the subcategories $\MM$ and $\VV$ from notation \ref{notation:RecoveringBialgebra} carries over to this setting.  As in lemma \ref{lemma:MMisModB}, we infer that $\MM \simeq \mod A.$  Under this equivalence, the functor $F = \Hom_{\glid FA}(P_0,-)\colon \MM \to \Vect_k$ corresponds to the forgetful functor $\Hom(A,-)\colon \mod A \to \Vect_k.$  We then recover $A$ as the algebra $\Hom(F,F).$
\end{proof}

As for bialgebras, there is no direct generalization of theorem \ref{theorem:RecoverAlgebra} to the setting of prefragments.

\begin{example}\label{example:ShortPrefragIsBlind}
Assume that $A$ and $A'$ be finite-dimensional $k$-algebras of the same dimension.  In this case, there is an equivalence of categories $\FF_\Lambda A \to \FF_\Lambda A'$ mapping $1 \in A$ to $1 \in A'$.  The natural functor $\Mod \FF_\Lambda A' \to \Mod \FF_\Lambda A$ induces an equivalence $\Prefrag FA' \to \Prefrag FA,$ and thus an equivalence $\prefrag FA' \to \prefrag FA,$ even though $A$ and $A'$ need not be isomorphic.
\end{example}

%% file: Appendix.tex
\section{Comparison to earlier works}

In this appendix, we provide a comparison of some notions in this paper to existing literature about glider representations.  Our main reference for this appendix is \cite{CaenepeelVanOystaeyenBook19}.

Fragments over filtered rings were first introduced in \cite{NawalVanOystaeyen95, NawalVanOystaeyen96} as a generalization of modules; since then, the definition has been amended (see \cite{CaenepeelVanOystaeyen19} or \cite{CaenepeelVanOystaeyenBook19}).  To avoid confusion, we opt to refer to the original concept (see \cite{NawalVanOystaeyen96}) as a prefragment, reserving the notion of a fragment for the one introduced in \cite{CaenepeelVanOystaeyenBook19}.

The definition of a glider as we use it, has been introduced in \cite{CaenepeelVanOystaeyen19}, as a prefragment $M$ whose partial actions $\phi\colon F_i R \times M_{-i} \to M$ are induced by an ambient $R$-module $\Omega_M$.  A morphism between glider representations (as given in \cite{CaenepeelVanOystaeyenBook19}) is required to be compatible with a chosen ambient $R$-module.  As expounded on in remark \ref{remark:Obstruction}, this compatibility condition obscures whether the composition of glider morphisms is well-defined.  Indeed, in \cite[\S1.7]{CaenepeelVanOystaeyenBook19}, a sequence of morphisms is called a glider sequence if the sequence is composable.

As the functor embedding $\phi\colon \Glid FR \to \Prefrag FR$ is fully faithful (see proposition \ref{proposition:PhiHasLeftAdjoint}), we can alternatively define the category of glider representations as follows.

\begin{definition}\label{definition:AppendixGliders}
Let $FR$ be a $\bZ$-filtered ring.  A \emph{glider representation} over $FR$ consists of an $R$-module $\Omega$ together with a $\bZ^{\leq 0}$-indexed chain sequence of subgroups:
\[\ldots \subseteq M_{-2} \subseteq M_{-1} \subseteq M_0 \subseteq \Omega\]
such that the action $R \otimes \Omega \to \Omega$ induces maps $F_i R \otimes M_{j} \to M_{i+j}.$

Let $M_\bullet \subseteq \Omega_M$ and $N_\bullet \subseteq \Omega_N$ be glider representations over $FR$.  A \emph{morphism of glider representations} $f\colon (M_\bullet \subseteq \Omega_M) \to (N_\bullet \subseteq \Omega_N)$ is given by an additive map $f\colon M_0 \to N_0$ satisfying the following conditions: 
  \begin{enumerate}
		\item for all $j \in \bZ^{\leq 0}$, we have $f(M_j) \subseteq N_j$, and
		\item\label{enumerate:GliderMorphism2} $\forall i \in \bZ$ and $\forall j \leq -i$, the following diagram commutes:
		\[
		\xymatrix{
		F_i R \otimes M_j \ar[r] \ar[d]^{1 \otimes f} & M_{i+j} \ar[d]^{f}\\
		F_i R \otimes N_j \ar[r] \ar[r] & N_{i+j}\\		}
		\]
		or, equivalently, for all $r \in F_i R$ and $m \in M_j$, we have $f(rm) = rf(m).$
	\end{enumerate}
\end{definition}

Following proposition \ref{proposition:GlidAsSubcatOfPrefrag}, a sequence $K \to M \to N$ is a \emph{conflation} if and only each of the induced sequences $0 \to K_j \to M_j \to N_j \to 0$ is exact in $\Mod F_0 R$ (for all $j \in \mathbb{Z}^{\leq 0}$).

\begin{remark}
\begin{enumerate}
	\item This definition is based on proposition \ref{proposition:PhiFullyFaithful}.  In this definition, the role of the $R$-module $\Omega_M$ is to determine whether the prefragment is, in fact, a glider representation; it plays no role in the definition of the morphisms, nor in determining whether a sequence of gliders is a conflation.
		\item As in \cite{CaenepeelVanOystaeyenBook19}, it suffices to require \eqref{enumerate:GliderMorphism2} for $i+j=0$.
	\item By taking $i=0$ in the definition of a glider morphism, we see that $f\colon M_j \to N_j$ is an $F_0R$-module homomorphism, for all $j \leq 0.$
	\item We do not claim that the map $f\colon M_0 \to N_0$ can be extended to an $R$-morphism $f_\Omega\colon\Omega_M \to \Omega_N$ (as is illustrated in example \ref{example:EasyGliders}), but it follows from proposition \ref{proposition:LeftAdjoint} that this can be done after possibly re-choosing $\Omega_M$.
	
	In fact, given the fragment part $\ldots \subseteq M_{-2} \subseteq M_{-1} \subseteq M_0$ of a glider, there is a canonical way of completing this to a glider in the sense of definition \ref{definition:AppendixGliders} by applying the functor $L\colon \Glid FR \to \Preglid FR$ (see \S\ref{subsection:GlidInPreglid}).
	\item The category $\Frag FR$ of fragments, considered in \cite{CaenepeelVanOystaeyenBook19}, is an additive subcategory of the category of prefragments, containing the category of glider representations.  Thus, $\Glid FR \subseteq \Frag FR \subseteq \Prefrag FR.$
\end{enumerate}
\end{remark}

\subsubsection*{Limits and colimits} Limits and colimits of fragments and gliders are discussed in \cite[\S1.8]{CaenepeelVanOystaeyenBook19}.  The categories of pregliders, prefragments, and (natural) gliders are complete and cocomplete.  With the description of a glider from definition \ref{definition:AppendixGliders}, limits are taken pointwise.  As is illustrated in example \ref{example:SymmetricGroup1}, colimits cannot be taken pointwise (see proposition \ref{proposition:LimitsAndColimits}).

\subsubsection*{The conflation structure on the category of glider representations}  A morphism in a conflation category is called \emph{admissible} if it admits a deflation-inflation factorization, i.e.~$f\colon X \to Y$ is admissible if it factors as $X \deflation Z \inflation Y.$ 	Admissible morphisms have been called \emph{strict} in \cite[proposition 1.7.1]{CaenepeelVanOystaeyenBook19}.  A strict monomorphism is an inflation, and a strict epimorphism is a deflation.

In particular, a morphism $f\colon M \to N$ of glider representations is a deflation if and only if, for all $i \in \bZ^{\leq 0}$, the map $f_i\colon M_i \to N_i$ is an epimorphism.
	
What is called the image of a morphism $f\colon X \to Y$ in \cite{CaenepeelVanOystaeyenBook19} is $\coker(\ker f)$, and is often called the coimage of $f$ (see, for example, \cite{Buhler10, Mitchell65}).  As the category of gliders is not an abelian category, the natural morphism from the coimage of a morphism to its image need not be an isomorphism.  However, as $\Glid FR$ is deflation quasi-abelian, a morphism $f\colon M \to N$ of gliders is the composition of a deflation $d\colon M \deflation \coim f$ and a monomorphism $m\colon \coim f \hookrightarrow N$ (see \cite[corollary 1]{Rump01}).

\subsubsection*{Noetherian glider representations} Noetherian fragments were introduced in \cite[\S2.3]{CaenepeelVanOystaeyenBook19}.  It follows from proposition \ref{proposition:WhenNoetherian} that a glider is Noetherian if and only if it is Noetherian as a prefragment.  Our theorem \ref{theorem:NoetherianObjects} and proposition \ref{proposition:WhenNoetherian} are analogues of \cite[theorems 2.3.2 and 2.3.4]{CaenepeelVanOystaeyenBook19}.

\subsubsection*{Projective glider representations}  A glider $M \in \Glid FR$ is projective if and only if the prefragment part $\phi(M) \in \Prefrag FM$ is projective.  A glider representation is freely generated in the sense of \cite[definition 2.1.3]{CaenepeelVanOystaeyenBook19} if and only if it is a direct sum of standard projectives.  A glider representation is projective if and only if it is a direct summand of a freely generated projective (a similar statement for fragments has been shown in \cite[proposition 2.2.5]{CaenepeelVanOystaeyenBook19}).

\subsubsection*{Monoidal structure for glider representations of filtered groups} Let $k$ be a field.  In \cite[proposition 4.6.7 and definition 4.6.8]{CaenepeelVanOystaeyenBook19}, a tensor product of $k$-linear fragments and glider representations of filtered groups were defined: the product $M \otimes N$ is defined via $(M \otimes N)_{-i} \cong M_{-i} \otimes_k N_{-i}.$  This coincides with the product considered in \S\ref{section:HopfCategories}.  In particular, for a filtered group $G_0 \subseteq G_1 \subseteq \ldots \subseteq G_n \subseteq \ldots$, the category of glider representations is (monoidally) derived equivalent to the category of representations of a semi-Hopf category.  Following \S\ref{subsection:Isocategorical}, one can recover the group $G$ from the monoidal category of glider representations.

\subsubsection*{Natural gliders} Let $\Omega$ be an $R$-module.  Given an $S$-submodule $M_0 \subseteq \Omega$, we can build a glider representation by setting $M_{-i}^* = \{m \in M_0 \mid F_i R \cdot m \subseteq M_0\}$.  Such glider representations (as well as those isomorphic) are called \emph{natural gliders} in \cite[definition 1.3.1]{CaenepeelVanOystaeyenBook19}.  In earlier work \cite{NawalVanOystaeyen95, NawalVanOystaeyen96}, the same concept was called a natural fragment.  The definitions given in \cite{CaenepeelVanOystaeyenBook19} and in \S\ref{section:NaturalGliders} coincide.  Our proposition \ref{proposition:InflationsToNaturalGliders} recovers \cite[lemma 1.6.2(3)]{CaenepeelVanOystaeyenBook19}.  It follows from theorem \ref{theorem:NaturalGliders} that the category $\NGlid FR$ of natural glider representations is a reflective subcategory of $\Glid FR.$  As such, the category $\NGlid FR$ is complete and cocomplete, and a limit of natural gliders is again a natural glider.  This answers a question posed at the end of \cite[\S1.8]{CaenepeelVanOystaeyenBook19}.